\documentclass[11pt,a4paper]{article}

\usepackage[T1]{fontenc}
\usepackage[utf8]{inputenc}
\usepackage[french]{babel}

\usepackage{amsmath, amsthm, amssymb}

\usepackage{graphicx, textcomp, lmodern, fullpage, url, tikz}

\usepackage[all]{xy}

\usepackage[ruled, vlined, linesnumbered, french]{algorithm2e}

\newcommand{\Z}{\mathbb{Z}}
\newcommand{\N}{\mathbb{N}}
\newcommand{\C}{\mathbb{C}}
\newcommand{\A}{\mathbb{A}}
\newcommand{\F}{\mathbb{F}}
\newcommand{\Q}{\mathbb{Q}}
\newcommand{\R}{\mathbb{R}}
\newcommand{\E}{\mathcal{E}}
\renewcommand{\Pr}{\mathcal{P}}
\newcommand{\Qr}{\mathcal{Q}}
\renewcommand{\H}{\mathbb{H}}
\renewcommand{\P}{\mathbb{P}}
\newcommand{\M}{\mathcal{M}}
\renewcommand{\O}{\mathcal{O}}
\newcommand{\Cl}{\mathcal{C}}
\renewcommand{\b}{\backslash}
\newcommand{\vers}{\longrightarrow}
\newcommand{\End}{\mathrm{End}}
\newcommand{\Hom}{\mathrm{Hom}}
\newcommand{\Ell}{\mathrm{Ell}}
\newcommand{\Jac}{\mathrm{Jac}}

\newcommand{\Spec}{\mathrm{Spec}\,}

\renewcommand{\frak}{\mathfrak}
\newcommand{\de}{\,:\,}

\renewcommand{\mod}{\ \mathrm{mod}\ }
\renewcommand{\v}{\vspace{5mm}}

\newtheorem*{thm}{Théorème}
\newtheorem*{lem}{Lemme}
\newtheorem*{prop}{Proposition}
\newtheorem*{cor}{Corollaire}
\newtheorem*{hyp}{Hypothèse}
\theoremstyle{definition}

\definecolor{darkgreen}{RGB}{0, 150, 0}

\begin{document}

\begin{titlepage}

\centering

{\Huge \'Etude et accélération du protocole d'échange de clés de Couveignes--Rostovtsev--Stolbunov}
\vspace{1cm}

{\large Mémoire du Master 2 Mathématiques fondamentales de l'Université Paris VI}
\vspace{1cm}

{\large Jean Kieffer}
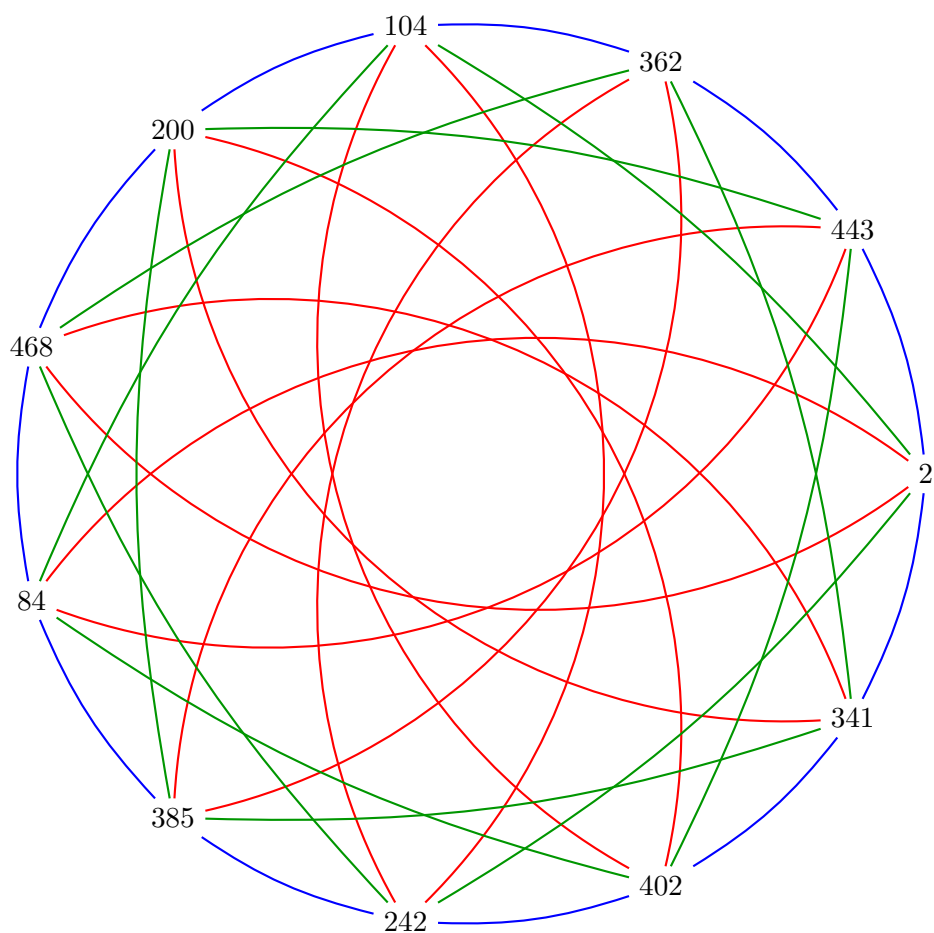
\v

\today

\vfill

\vfill

\begin{tikzpicture}[x=6cm, y=6cm]

\draw[]
 (1, 0) node(7) {2}
 (0.84, 0.54) node(443) {443}
 (0.42, 0.91) node(362) {362}
 (-0.14,0.99) node(104) {104}
 (-0.65,0.76) node(200) {200}
 (-0.96,0.28) node(468) {468}
 (-0.96, -0.28) node(84) {84}
 (-0.65, -0.76) node(385) {385}
 (-0.14, -0.99) node(242) {242}
 (0.42, -0.91) node(402) {402}
 (0.84, -0.54) node(341) {341};
\draw[]
 (7) edge[thick, blue, bend right = 10] (443)
 (443) edge[thick, blue, bend right = 10] (362)
 (362) edge[thick, blue, bend right = 10] (104)
 (104) edge[thick, blue, bend right = 10] (200)
 (200) edge[thick, blue, bend right = 10] (468)
 (468) edge[thick, blue, bend right = 10] (84)
 (84) edge[thick, blue, bend right = 10] (385)
 (385) edge[thick, blue, bend right = 10] (242)
 (242) edge[thick, blue, bend right = 10] (402)
 (402) edge[thick, blue, bend right = 10] (341)
 (341) edge[thick, blue, bend right = 10] (7);

\draw[]
 (7) edge[thick, red, bend left=45] (468)
 (468) edge[thick, red, bend left=45] (341)
 (341) edge[thick, red, bend left=45] (200)
 (200) edge[thick, red, bend left=45] (402)
 (402) edge[thick, red, bend left=45] (104)
 (104) edge[thick, red, bend left=45] (242)
 (242) edge[thick, red, bend left=45] (362)
 (362) edge[thick, red, bend left=45] (385)
 (385) edge[thick, red, bend left=45] (443)
 (443) edge[thick, red, bend left=45] (84)
 (84) edge[thick, red, bend left=45] (7);
 
\draw[]
 (7) edge[thick, darkgreen, bend left=10] (242)
 (242) edge[thick, darkgreen, bend left=10] (468)
 (468) edge[thick, darkgreen, bend left=10] (362)
 (362) edge[thick, darkgreen, bend left=10] (341)
 (341) edge[thick, darkgreen, bend left=10] (385)
 (385) edge[thick, darkgreen, bend left=10] (200)
 (200) edge[thick, darkgreen, bend left=10] (443)
 (443) edge[thick, darkgreen, bend left=10] (402)
 (402) edge[thick, darkgreen, bend left=10] (84)
 (84) edge[thick, darkgreen, bend left=10] (104)
 (104) edge[thick, darkgreen, bend left=10] (7);

\end{tikzpicture}

\vfill

\end{titlepage}

\begin{abstract}

L'objet de ce document est l'étude d'un protocole d'échange de clés fondé sur l'utilisation d'isogénies entre courbes elliptiques ordinaires sur un corps fini, proposé par Couveignes \cite{Couv} puis étudié par Rostovtsev et Stolbunov \cite{RoSt}. Après une courte présentation des notions fondamentales sur les courbes elliptiques, on présente la théorie de la multiplication complexe sur $\C$, puis son adaptation au corps finis sur laquelle le cryptosystème est construit. On introduit également les courbe modulaires, dont la manipulation est nécessaire dans les calculs. Après une exposition du protocole et des algorithmes utilisés, on présente une idée originale pour l'accélération du cryptosystème. La discussion de l'implémentation et des résultats pratiques concluent ce mémoire.

\end{abstract}

\vfill

\tableofcontents

\vfill

\newpage

\section*{Introduction}

\addcontentsline{toc}{section}{Introduction}

Le sujet de ce mémoire rentre dans le cadre de la \emph{cryptographie à base d'isogénies} entre courbes elliptiques, qui est un domaine très récent de la cryptographie asymétrique, et s'inscrit dans la lignée de la cryptographie sur courbes elliptiques.

Le principe de cette dernière est le suivant: les courbes elliptiques sont des groupes algébriques, et sont donc utilisables comme brique de base pour de nombreux protocoles en cryptographie asymétrique. La plupart du temps, on utilise une courbe elliptique en cryptographie parce que le problème du logarithme discret pour les points de cette courbe est réputé difficile (plus difficile que dans le groupe multiplicatif d'un corps fini, par exemple). Dans ce cadre, faire intervenir des isogénies entre courbes elliptiques est une idée récente, qui date du début des années 2000. À l'origine, ces morphismes entre courbes sont plutôt utilisés de manière destructrice, pour ramener une courbe elliptique donnée vers une courbe elliptique plus faible, c'est à dire plus facile à casser. On a ensuite commencé à utiliser des isogénies de manière constructive, en proposant des protocoles cryptographiques les utilisant. La tête de pont de ce mouvement est le protocole SIDH \cite{SIDH}, où interviennent des isogénies de petit degré entre courbes elliptiques \emph{supersingulières} sur un corps fini.

Le point de départ de ce projet de recherche est une proposition de Jean-Marc Couveignes \cite{Couv} d'utiliser une \emph{action} d'un groupe de classes sur un ensemble de courbes elliptiques pour construire un protocole d'échange de clés. En quelque sorte, il s'agit de l'analogue du système SIDH pour des courbes elliptiques \emph{ordinaires}. Cette idée a été étudiée par Rostovtsev et Stolbunov \cite{RoSt}, qui décrivent ce protocole plus en détail et en donnent des exemples jouets. L'objectif de ce document est d'étudier ce protocole cryptographique, et d'être capable de comprendre précisément quels doivent être ses paramètres et quelle est son efficacité réelle. On étudie également une proposition originale d'accélération de ce protocole.

\v

Avant de pouvoir décrire ce cryptosystème, il faut comprendre les propriétés de structure des isogénies entre courbes elliptiques sur un corps fini, données par la théorie de la \emph{multiplication complexe}. C'est l'objet de la première partie de ce document, qui introduit les propriétés et les concepts fondamentaux pour la suite. Lorsque l'on parle d'isogénies entre des courbes elliptiques, on est également rapidement amené à parler de \emph{courbes modulaires}, qui sont des espaces naturels paramétrisant les courbes elliptiques munies d'une certaine structure. Leur description fait l'objet de la deuxième partie, d'abord sur les nombres complexes où elles sont étroitement reliées aux formes modulaires, puis de façon plus générale.

Après ces préparations, on présente le protocole de Couveignes--Rostovtsev--Stolbunov lui-même et le calcul d'isogénies entre courbes elliptiques, d'un point de vue algorithmique. Le but de la troisième partie est d'introduire différents algorithmes utilisés pour ce calcul, et de discuter leur coût; on y apprend également ce que représente ce mystérieux graphe de couverture.

Avant d'étudier les performances de ce protocole cryptographique, certains de ses paramètres doivent être déterminés, et ceux-ci dépendent de la qualité des attaques disponibles contre le protocole. Cette étude est faite en quatrième partie, où l'on montre également que la sécurité du cryptosystème est étroitement liée à des questions fines de structure de certains groupes de classes. De nombreuses réponses à ces questions d'arithmétique restent aujourd'hui inconnues. On discute également de la recherche d'une bonne courbe initiale, dont dépend directement la possibilité d'accélérer le protocole.

Enfin, ce stage a été l'occasion d'implémenter le protocole de Couveignes--Rostovtsev--Stolbunov à l'aide de logiciels de calcul formel. Ce travail, chronophage lorsque l'on débute en programmation sérieuse, est décrit dans la cinquième et dernière partie du document. On obtient une mesure précise du coût de différentes étapes critiques des algorithmes évoqués ci-dessus, ce qui permet de se faire une idée de leur coût total, et permet également de désigner les meilleures courbes initiales. À terme, ce code est destiné à devenir un module open-source pour le logiciel de calcul formel Nemo contenant différentes primitives utiles en cryptographie sur courbes elliptiques.

\v

Ce mémoire a été réalisé au cours du stage de second semestre du Master 2 de Mathématiques fondamentales, à l'université Paris VI. Ce stage s'est déroulé entre mars et août 2017, au sein de l'équipe GRACE (Geometry, Arithmetic, Codes and Encryption) de l'organisme de recherche publique français en mathématiques et informatique, Inria. Une partie du contenu de ce mémoire a fait l'objet d'un poster présenté à la conférence ISSAC (International Symposium on Symbolic and Algebraic Computation) qui a eu lieu au mois de juillet 2017 à Kaiserslautern. Celui-ci a été récompensé du Best Poster Award à cette occasion.

Un grand merci à Daniel Augot et aux membres de l'équipe Grace pour leur excellent accueil, à Jessica Gameiro pour le soutien logistique, à Benoît Stroh, Andrew Sutherland et Jean-Pierre Flori pour des échanges de mails éclairants. Enfin, bien sûr, merci à Luca De Feo et Benjamin Smith pour leur encadrement bienveillant, pour leurs réponses à mes nombreuses questions, pour la relecture de ce document et pour m'avoir fait découvrir un domaine passionnant entre arithmétique, géométrie algébrique, calcul formel et cryptographie.

\newpage

\section{Courbes elliptiques}

\subsection{Définitions et propriétés basiques}

Une \emph{courbe elliptique} $E$ sur un corps $k$ est une courbe algébrique propre et lisse définie sur $k$ et de genre 1, munie d'un $k$-point fixé $\O_E$. Un exemple typique est donné par les \emph{équations de Weierstrass}:
$$y^2 + a_1xy = x^3 + a_2x^2 + a_4 x + a_6.$$
Cette courbe affine plane, lorsqu'elle est lisse (ce qui est équivalent à la non-nullité d'un certain discriminant, qui est une fonction polynomiale des $a_i$), peut être complétée en une courbe projective en ajoutant un unique point à l'infini. On obtient alors une courbe elliptique en prenant ce dernier point comme \og origine\fg\ de la courbe. Le $j$-invariant d'une courbe elliptique est une fraction rationnelle des $a_i$, et deux courbes elliptiques sont isomorphes sur $\bar{k}$ si et seulement si elles ont même $j$-invariant.

Soit $E$ une courbe elliptique sur $k$ dans le sens donné ci-dessus. Comme à toute courbe algébrique lisse, on peut lui attacher un schéma en groupes abélien défini sur $k$, sa \emph{jacobienne}. Un schéma en groupes (qui sera toujours abélien dans ce document) est simplement un schéma dont les points forment toujours un groupe abélien, c'est à dire un ensemble d'équations algébriques (ici à coefficients dans $k$) dont les solutions forment toujours un groupe. $GL_n$, le groupe des matrices carrées de déterminant inversible, ou $\mu_n$, le groupe des racines $n$-ièmes de l'unité, en sont des exemples qui sont définis sur $\Z$, donc aussi sur $k$. On peut définir les $\bar{k}$-points de $\Jac(E)$ de la manière suivante:
\begin{itemize}
\item[•] Un \emph{diviseur} de $E$ est une somme formelle de points de $E(\bar{k})$ à coefficients dans $\Z$.
\item[•] Le \emph{degré} d'un diviseur est la somme de ses coefficients.
\item[•] Le diviseur d'une fonction sur $E$ à valeurs dans $k$ (non nulle) est la somme de ses zéros et pôles avec multiplicité, et est systématiquement de degré zéro. Ces diviseurs sont dits \emph{principaux}.
\item[•] Le quotient du groupe des diviseurs de degré zéro par celui des diviseurs principaux forme le groupe des $\bar{k}$-points de la jacobienne de $E$.
\end{itemize}
Le théorème de Riemann--Roch, un outil puissant dans l'étude des courbes algébriques, permet de contrôler le nombre de fonctions sur $E$ dont le diviseur satisfait une certaine condition, qui se traduit en des contraintes de zéros et de pôles sur ces fonctions. Il implique que l'on dispose d'un unique isomorphisme entre la courbe $E$ et sa jacobienne envoyant $\O_E$ sur zéro, que l'on peut donner sur les points de la manière suivante:
$$\begin{aligned}
E(\bar{k}) &\overset{\sim}{\vers} \Jac(E)(\bar{k}) \\
P &\longmapsto \text{Classe de } (P) - (\O_E).
\end{aligned}$$
L'idée est la suivante: soit $D$ un diviseur de degré zéro. Alors $-D - (\O_E)$ est de degré $-1$, et il existe essentiellement une unique fonction $f$ dont les pôles et zéros sont contrôlés par ce diviseur, par le théorème de Riemann-Roch. Par la condition de degré, $f$ a un unique zéro supplémentaire, $P$. Le diviseur de $f$ est donc $- D - (\O_E) + (P)$, ce qui montre que $D = (P) - (\O_E)$ dans la jacobienne. On a donc montré la surjectivité de l'application ci-dessus.

Cet isomorphisme permet de munir $E$ d'une structure de $k$-schéma en groupes, en particulier de munir l'ensemble $E(L)$ des $L$-points de $E$ d'une structure de groupe pour toute extension $L/k$, dont l'élément neutre est $\O_E$. Le théorème de Riemann--Roch implique également que toute courbe elliptique sur $K$ admet une équation plane sous forme de Weierstrass, en regardant cette fois les fonctions ayant un pôle uniquement en $\O_E$: la fonction $x$ de l'équation plane précédente y admet un pôle double, $y$ un pôle triple, et l'équation polynomiale provient du fait qu'il y a trop de fonctions ayant un pôle d'ordre au plus 6 en ce point. Bien sûr, ce modèle n'a rien de canonique et on peut donner une courbe elliptique par d'autres équations. Si la caractéristique de $k$ est différente de 2 et 3, on peut de plus effectuer quelques changements de coordonnées et choisir $a_1 = a_2 = a_3 = 0$, et l'on obtient alors une forme de Weierstrass dite \emph{réduite} ou \emph{courte}.

\v

Une \emph{isogénie} de $E$ vers $E'$ est une application rationnelle surjective (ou, de manière équivalente, non nulle) envoyant $\O_E$ sur $\O_{E'}$, et est automatiquement un morphisme de groupes. 
En ajoutant l'application nulle $x\mapsto \O_{E'}$, on obtient les \emph{morphismes} entre $E$ et $E'$. Notons qu'un morphisme entre deux courbes elliptiques sur $k$ n'est pas nécessairement lui-même défini sur $k$. Par exemples, deux courbes elliptiques peuvent être isomorphes sur une extension quadratique de $k$ mais pas sur $k$ lui-même. Elles sont alors dites \emph{tordues} l'une de l'autre.

Notons $\End(E)$ l'anneau d'endomorphismes de $E$ : la structure de groupe donne une flèche
$$\begin{aligned}
&\Z &\longrightarrow&\ &\End(E) &\\
&n &\longmapsto& &[n]_E .\ \ &
\end{aligned}$$
Pour en dire plus sur la structure de $\End(E)$, on peut montrer que pour toute isogénie $\phi\de E\vers E'$ de degré $m$, il existe une unique isogénie notée $\widehat{\phi}\de E'\vers E$ de degré $m$ telle que $\phi\widehat{\phi}=[m]_{E'}$ et $\widehat{\phi}\phi=[m]_{E}$, que l'on appelle \emph{duale} de $\phi$. L'isogénie duale d'une somme est la somme des isogénies duales. On en déduit que le degré est une forme quadratique sur $\End(E)$ et $[m]_E$ est de degré $m^2$, donc la flèche précédente est injective.

On a alors trois possibilités: $\End(E)$ est soit
\begin{enumerate} 
\item[•]$\Z$, 
\item[•]un ordre dans un corps quadratique imaginaire,
\item[•]un ordre maximal dans une algèbre de quaternions sur $\Q$. 
\end{enumerate}
Un \emph{ordre} est un sous-anneau qui est de type fini sur $\Z$ et engendre tout l'espace comme $\Q$-espace vectoriel. Par exemple, l'anneau des entiers $\O_K$ d'un corps quadratique imaginaire $K$ (on peut penser à $\Z[i]$ dans $\Q[i]$) est un ordre de $K$. C'est d'ailleurs l'ordre maximal, et tous les ordres de $K$ sont des sous-ordres de $\O_K$. Cette alternative à trois choix est le mieux que l'on puisse faire dans le cas général, mais les choses s'arrangent un peu lorsqu'on regarde deux cas particuliers: les corps finis d'une part, et la caractéristique nulle d'autre part. Dans l'optique d'utiliser des courbes elliptiques en cryptographie, ce sont les premières qui nous intéressent.

\v

\textbf{Courbes elliptiques sur un corps fini.} Reprenons les trois cas ci-dessus. Si $k = \F_p$, où $p$ est un nombre premier, on dit que $E$ est \emph{ordinaire} dans le second cas et que $E$ est \emph{supersingulière} dans le troisième. Le premier cas est exclu par la présence du morphisme de Frobenius
$$\pi_E\de (x,y)\mapsto (x^p, y^p)$$
qui n'est pas un élément de $\Z$ (son degré, $p$, n'est pas un carré). On vient d'écrire le Frobenius pour une équation plane de $E$, mais il existe une définition plus intrinsèque, l'idée générale restant de tout mettre à la puissance $p$.

La théorie de Hasse nous dit que $\pi_E$ vérifie une équation de la forme
$$\pi_E^2 - t \pi_E + p = 0.$$
L'autre élément de $\End(E)$ vérifiant cette équation est le dual du Frobenius.
Le nombre $t$, appelé \emph{trace} de $\pi_E$ (ou de la courbe $E$), est un entier vérifiant $|t|\leq 2\sqrt{p}$. On peut montrer que le cardinal de l'ensemble $E(\F_p)$ est donné par
$$\# E(\F_p) = p + 1 - t,$$
d'où les \emph{bornes de Hasse} :
$$p + 1 - 2\sqrt{p} \leq \#E(\F_p) \leq p + 1 + 2\sqrt{p}$$
qui jouent un rôle important en pratique dans le calcul de $\# E(\F_p)$. On peut facilement lire sur le cardinal de $E(\F_p)$ le caractère supersingulier ou ordinaire de $E$: la courbe $E$ est supersingulière si et seulement si on a l'égalité $t = 0 \mod p$, ce qui est équivalent à $t=0$ dès que $p > 2\sqrt{p}$, c'est à dire $p \geq 5$. On peut montrer que ces résultats sont en fait valables sur tout corps fini, en remplaçant $p$ par $q$, une puissance de nombre premier.

\v

\textbf{Courbes elliptiques en caractéristique zéro}. Tout d'abord, le principe de plongement de Lefschetz permet de ramener les corps de caractéristique zéro au cas complexe. Si $E/\C$ est une courbe elliptique, c'est en particulier une surface de Riemann. On peut montrer qu'il existe un nombre complexe $\tau$ dans le demi-plan de Poincaré $\H$ tel que l'on ait un isomorphisme
$$\C/\Lambda_\tau \overset{\sim}{\vers} E.$$
On a noté ici $\Lambda_\tau = \Z\oplus \Z\tau$, et cet isomorphisme est donné par la fonction $\wp$ de Weierstrass et sa dérivée. Le nombre $\tau$ est unique modulo l'action de $\Gamma(1) = \mathrm{SL}_2(\Z)$ sur $\H$ par homographies. Dans cette description, une fonction sur $E$ à valeurs complexes est simplement une fonction méromorphe sur le plan complexe admettant deux périodes, données par les deux vecteurs d'une base de $\Lambda_\tau$. D'un point de vue historique, l'étude de ces \emph{fonctions elliptiques} (dont un membre éminent est la longueur d'arc d'une ellipse) est le point de départ de celle des courbes elliptiques elles-mêmes. C'est là l'origine d'un nom peut-être malheureux, car une courbe elliptique et une ellipse n'ont pas grand-chose en commun.

Dans ce contexte, une isogénie $E\vers E'$ est simplement la multiplication par un $\alpha\in \C^*$ tel que $\alpha \Lambda_\tau \subset \Lambda_\tau'$. Avec cette description explicite, il est facile de vérifier les différentes propriétés exposées plus haut: par exemple, le dual d'une isogénie est simplement son conjugué complexe. On voit également que $\End(E)$ ne peut jamais être un ordre d'une algèbre de quaternions. Tout cela motive une terminologie différente du cas des corps finis: on dira que $E$ est \emph{à multiplication complexe} si $\End(E)$ est un ordre quadratique imaginaire.
On utilise également le terme de multiplication complexe pour parler de courbes elliptiques ordinaires sur un corps fini, leurs propriétés étant globalement similaires.

Cette description agréable en termes de réseaux ne tient plus en caractéristique positive, et le passage à celle-ci semble toujours faire intervenir une preuve non triviale. Dans ce mémoire, on se servira du cas complexe comme une motivation pour la démonstration des résultats généraux.

\v

Le matériel aperçu dans cette introduction est étudié plus en profondeur (et plus proprement) dans plusieurs cours sur les courbes elliptiques, comme \cite{Nekovar} et \cite{Stroh}.
Le but de ce document étant de travailler avec des courbes elliptiques ordinaires sur un corps fini, on commence par étudier la théorie de la multiplication complexe sur $\C$.

\subsection{Multiplication complexe sur $\C$}

La référence principale de cette section est \cite{Sil2}. On fixe ici un ordre $\O$ dans un corps quadratique imaginaire $K\subset \C$, une extension de $\Q$ de degré 2 vérifiant $K \not\subset\R$: concrètement, on a $K=\Q(\sqrt{-d})$ pour un certain entier $d\geq 2$ sans facteur carré. On note $\Ell_\C(\O)$ l'ensemble des courbes elliptiques sur $\C$ (à isomorphisme près) ayant multiplication complexe par $\O$, c'est à dire les courbes $E_\tau = \C/\Lambda_\tau$ pour lesquelles on a:
$$\End(E_\tau) = \{\alpha\in \C\ |\ \alpha\Lambda_\tau \subset \Lambda_\tau\} = \O.$$
On identifiera un élément de cet ensemble avec ses représentants. Dans \cite{Sil2}, on suppose que $\O$ est l'anneau d'entiers de $K$, ce qui permet de simplifier certaines démonstrations; cette supposition n'est pas essentielle en réalité, et on énonce ici les résultats pour un ordre quelconque.

Si $\O_K$ désigne l'anneau d'entiers de $K$, on appelle $f=[\O:\O_K]$ le \emph{conducteur} de $\O$, et l'on a $\O = \Z  + f \O_K$. Le discriminant de $\O$ est alors $D = f^2 D_K$. Les idéaux de $\O$ premiers à $f$ sont les idéaux inversibles de $\O$ et engendrent un groupe d'idéaux fractionnaires, que l'on quotiente par les idéaux principaux pour obtenir le \emph{groupe de classes d'idéaux} $\Cl(\O)$ de $\O$.
\v

La première étape est de fixer un isomorphisme $[\,\cdot\,]\de \O\vers \End(E)$ pour toute courbe $E\in \Ell_\C(\O)$ afin de pouvoir identifier ces anneaux. On fait cela en regardant l'action sur le $\C$-espace vectoriel de dimension 1 des formes différentielles sur $E$, en imposant
$$\forall \omega\in H^0(E,\Omega^1),\ \forall\,\alpha\in\O,\ [\alpha]^*\omega = \alpha \omega.$$
Cela coïncide bien avec la description précédente: sur $E_\tau = \C/\Lambda_\tau$, la différentielle $dz$ est invariante, et l'on a bien $\alpha^* dz = \alpha dz$. Cette condition est suffisante pour fixer un unique isomorphisme, car l'action sur la différentielle invariante n'est nulle que si l'endomorphisme en question est nul. Cela provient du fait que tous les morphismes sont \emph{séparables} en caractéristique zéro.

Le but est maintenant de définir une action de $\Cl(\O)$ sur ces courbes elliptiques. Au niveau des idéaux, on procède de la manière suivante: soit $E\in \Ell_\C(\O)$. En raisonnant en termes de réseaux, on voit que pour tout sous-groupe fini $S$ de $E(\C)$, il existe une courbe elliptique $E'$ et une isogénie $E\vers E'$ de noyau $S$, et celle-ci est unique à isomorphisme près (on déclare que deux isogénies sont isomorphes si elles diffèrent d'un isomorphisme à l'arrivée). Soit $\frak a$ un idéal inversible de $\O$. On définit
$$E[\frak a]=\bigcap_{\phi\in \frak a} \mathrm{Ker}\,\phi$$
(rappelons l'identificiation $\End(E)\simeq\O$ ci-dessus). On note $\phi_{\frak a}$ l'isogénie de noyau $E[\frak a]$, et $\frak a\cdot E$ sa courbe image, qui est bien définie comme élément de $\Ell_\C(\O)$.

\v

Pour montrer que cette \emph{action par isogénies} en est bien une, on peut faire appel une fois de plus aux réseaux et montrer la propriété suivante : pour tout idéal inversible $\frak a$ de $\O$ et toute courbe $E=\C/\Lambda\in \Ell_\C(\O)$, on a
$$\frak a\cdot E = \C/\frak a^{-1} \Lambda.$$
On en déduit que le degré de $\phi_{\frak a}$ est la norme de l'idéal $\frak a$, et que la courbe $\frak a\cdot E$ a également multiplication complexe par $\O$. On en déduit aussi la transitivité, et il est immédiat que les idéaux principaux agissent trivialement, l'isogénie en question étant alors l'endomorphisme qui engendre l'idéal. On obtient donc l'action annoncée
$$\Cl(\O) \circlearrowright \Ell_\C(\O).$$

De plus, cette dernière action est simplement transitive. Les preuves données dans \cite{Sil2} utilisent de manière essentielle la représentation des courbes elliptiques sur $\C$ sous la forme $\C/\Lambda$, et tout le jeu consiste à dire que $\Lambda$ est déjà essentiellement un idéal fractionnaire de $\O$. Cette méthode n'est donc pas transposable telle quelle en caractéristique positive. Pourtant, un résultat similaire reste valable pour les courbes elliptiques sur un corps fini.

\subsection{Multiplication complexe sur les corps finis}

 Dans cette section, $k$ désigne un corps fini. On peut définir comme précédemment $\frak a\cdot E$ pour $\frak a$ un idéal inversible de $\O$ et $E\in \Ell_k(\O)$, comme l'image de l'isogénie de noyau $E[\frak a]$, à l'aide de la propriété suivante.

\begin{lem}[Quotient, cas étale] Soit $S$ un sous-groupe fini de $E(\bar{k})$. Il existe une courbe elliptique $E'$ définie sur $\bar{k}$ et une isogénie séparable $\phi_S\de E\vers E'$ de noyau $S$. Cette isogénie est unique à isomorphisme près. Si $S$ est défini sur $k$, c'est à dire globalement stable par l'action du groupe de Galois, alors $\phi_S$ et $E'$ peuvent également être définies sur $k$.
\end{lem} 

Ce lemme est en fait vrai dans un cadre plus général qui permet de prendre en compte le cas des isogénies inséparables, où $S$ est un sous-schéma en groupes fini et plat de $E$. Ici, la courbe $\frak a\cdot E$ reste donc définie sur $k$, puisque $E[\frak a]$ l'est. On peut donner deux stratégies pour obtenir ce lemme: l'une très concrète, en partant d'une équation de la courbe et du noyau et en construisant l'équation de la courbe image, et l'autre plus abstraite, en construisant la courbe image comme un quotient dans le monde des schémas. On reviendra sur la première méthode plus loin dans ce document.

\begin{prop}

Lorsque $k$ est un corps fini, on dispose également d'une action de $\Cl(\O)$ sur $\Ell_k(\O)$ qui est simplement transitive.

\end{prop}

C'est ce résultat, que l'on pourrait appeler \og théorème principal de la multiplication complexe sur les corps finis\fg, qui est à la base du cryptosystème de Couveignes, Rostovtsev et Stolbunov. Le reste de cette section est consacré à la preuve de ce résultat, qui est le dernier outil nécessaire à la construction du protocole cryptographique.

\subsection{Au sujet de la preuve du théorème principal}

Une première idée pour atteindre les corps finis est de quotienter un anneau dans un corps de nombres par un idéal. Partant d'une courbe elliptique définie sur $\F_p$, on se demande par exemple si l'on peut considérer un relèvement à $\Z$, pour lequel les résultats sont connus vu que $\Z\subset\C$, et redescendre tout cela ensuite. 

Le problème inverse est de partir d'une courbe elliptique définie sur $\Q$ et étudier ses réductions modulo $p$. Si l'on se donne par exemple une courbe elliptique $E$ définie sur $\Q$, alors $E$ admet un modèle à coefficients dans $\Z$. Pour réduire $E$ en caractéristique $p$, on cherche d'abord à éliminer toutes les puissances possibles de $p$ pour trouver un \emph{modèle minimal} en $p$. Deux cas se présentent alors : soit la courbe réduite modulo $p$ est lisse (c'est alors une courbe elliptique et on dit que $E$ a \emph{bonne réduction} en $p$), soit elle ne l'est pas (elle dégénère en un n\oe ud ou une pointe et on dit que $E$ a \emph{mauvaise réduction}). Étudier la réduction de courbes définies sur $\Q$, ou plus généralement sur un corps de nombres est une question d'arithmétique vaste.

\v

Cette question de relèvement est ancienne, et on dispose des théorèmes suivants, dus à Deuring et dont une preuve figure dans \cite{Lang}. On se donne un corps fini $k$ de caractéristique $p$, et on note $x\mapsto\bar{x}$ les réductions modulo $p$.

\begin{thm}[Relèvement d'un endomorphisme en caractéristique zéro]

Soit $E/k$ une courbe elliptique et $\alpha$ un endomorphisme de $E$. Alors il existe un corps de nombres galoisien $L$, une courbe elliptique $E^0/L$, un endomorphisme $\alpha^0$ de $E^0$ et un premier $\frak p$ de $L$ au-dessus de $p$, tels que $E^0$ a bonne réduction en $\frak p$, $E$ est isomorphe à $\bar{E^0}$ et $\alpha$ correspond à $\bar{\alpha^0}$ sous cet isomorphisme.

\end{thm}

En particulier, si $\End(E)$ est isomorphe à un ordre $\O$, on peut choisir pour $\alpha$ un générateur de $\End(E)$ et on obtient une surjection $\End(E^0)\vers \End(E).$ Le résultat suivant montre en fait que dans ce cas, on obtient une bijection entre les anneaux d'endomorphismes.

\begin{thm}[Bonne réduction]

Soit $L$ un corps de nombres, $E/L$ une courbe elliptique telle que $\End(E)\simeq \O$ est un ordre dans un corps quadratique imaginaire $K$. Soit $p$ un nombre premier et $\frak p$ un premier de $L$ au-dessus de $p$, en lequel $E$ a bonne réduction. Alors $\bar{E}$ est ordinaire si et seulement si $p$ est totalement scindé dans $K$. Dans ce cas, si $c=p^r c_0$ est le conducteur de $\O$, avec $c_0 \wedge p = 1$, on a :

\begin{itemize}
\item[(i)] $\End(\bar{E})=\Z+c_0 \O_K$ est l'ordre de $K$ de conducteur $c_0$.
\item[(ii)] Si $c=c_0$, alors la réduction donne un isomorphisme $\End(E)\vers\End(\bar{E})$.
\end{itemize}
De plus, $\End(E)\vers\End(\bar{E})$ préserve le degré.

\end{thm}

On fixe également un ordre $\O$ dans un corps quadratique imaginaire, que l'on voit dans $\C$. Si $\O$ est l'anneau d'endomorphismes d'une certaine courbe définie sur $k$, alors par le théorème de bonne réduction, on a une égalité d'idéaux
$$(p) = \frak{p}\, \sigma( \frak{p})$$
dans $\O$, où $\sigma$ désigne la conjugaison complexe. Si $E$ a multiplication complexe par $\O$, il existe exactement deux isomorphismes entre $\End(E)$ et $\O$, l'un envoyant le Frobenius de $E$ dans $\frak p$, l'autre dans $\sigma(\frak p)$. On convient que quitte à faire agir $\sigma$, on considère le Frobenius comme un élément de $\frak p$. C'est l'analogue de l'isomorphisme canonique de la section précédente (notons d'ailleurs que la réduction modulo $\frak p$ rend les deux conventions compatibles).

Tentons donc de démontrer avec ces outils le théorème principal de la multiplication complexe sur les corps finis.

\begin{proof}

Soit $E$ une courbe elliptique définie sur $k$, et $E^0$ un relèvement de $E$ en caractéristique zéro donné par les théorèmes de Deuring. Soit également $\frak a$ un idéal de $\O$. Comme $E^0$ a multiplication complexe par $\O$, l'idéal $\frak a$ agit sur $E^0$:
$$E^0 \overset{\phi_{\frak a}^0}{\vers} E'^0.$$
On peut montrer que $E'^0$ et $\phi_{\frak a}^0$ sont définis sur une extension finie de $\Q$.

Soit $L$ l'extension de $\Q$ et $\frak p$ la place de $L$ donnés par le théorème de relèvement. Notons $M$ une extension de $L$ sur laquelle $E^0$, $E'^0$ et $\phi_{\frak a}^0$ sont définis, et soit $\frak P$ une place de $M$ au-dessus de $\frak p$. Le quotient $M/\frak P$ est donc une extension finie de $k = L/\frak p$. Comme on sait que $E^0$ a bonne réduction en $\frak p$, on peut tout réduire modulo $\frak P$, et on obtient
$$E \overset{\phi}{\vers} E'$$
où $\phi$ est définie sur une extension finie de $k$. Le degré de $\phi$ est égal à celui de $\phi_{\frak a}^0$, dont le noyau est $E^0[\frak a]$, et le noyau de $\phi$ contient les réductions à $k$ des éléments de $E^0[\frak a]$. 

Nous admettons maintenant un ingrédient essentiel, à savoir le fait suivant. Pour tout idéal $\frak a$ inversible de $\O$ de norme première à $p$, et toute courbe $E\in \Ell_k(\O)$, on peut choisir un relèvement $E^0$ vérifiant les conditions suivantes:
\begin{itemize}
\item[•] Les éléments de $E[\alpha]$ sont entiers en $\frak p$
\item[•] La réduction modulo $\frak p$ donne une bijection $E^0[\frak a] \overset{\sim}{\to} E[\frak a]$.
\end{itemize}

Avec ce fait admis, on montre donc que le noyau de $\phi$ est précisément $E[\frak a]$. Autrement dit, $\phi$ est, à isomorphisme près, l'isogénie $\phi_{\frak a}$, et on a $E' = \frak a\cdot E$. La courbe $\frak a\cdot E$ est donc la réduction de $\frak a\cdot E^0$. D'autre part, on peut utiliser le théorème de bonne réduction pour voir que $\frak a\cdot E$ a, comme $E'^0$, multiplication complexe par l'ordre $\O$. On en déduit que l'on a une action
$$\Cl(\O) \circlearrowright \Ell_k(\O).$$

Montrons qu'elle est transitive : si $E$ et $E'$ sont données en caractéristique $p$, choisissons deux relèvements $E^0, E'^0$ en caractéristique 0. Il existe alors un idéal $\frak a$ envoyant $E^0$ sur $E'^0$, par transitivité dans $\C$. En réduisant on voit que $\phi_{\frak a}$ a bien pour image $E'$.

Enfin, étant donné $E$, montrons que son stabilisateur est l'ensemble des idéaux principaux. Si $\frak a$ est principal, il laisse un relèvement $E^0$ invariant, et donc $E$ également. Réciproquement, si $\frak a$ laisse $E$ invariante, alors $\phi_{\frak a}$ admet un relèvement $\psi\in \End(E^0)$. Comme $\psi$ et $\phi_{\frak a}^0$ ont même degré, on voit que ces isogénies ont même noyau, ce qui montre $\psi=\phi_{\frak a}^0$ (modulo un isomorphisme) et $\frak a\cdot E^0 = E^0$. Ainsi $\frak a$ est principal, ce qui termine la preuve.
\end{proof}
\v

Cette preuve \og trouée\fg\ semble raisonnablement directe, mais nous ne sommes pas parvenus à la compléter, ni à en trouver une complétion dans la littérature. On peut simplement remarquer que le fait admis est vrai lorsque $\frak a$ est l'idéal engendré par un entier $n$ premier à $p$, ce qui est plutôt inutile pour notre preuve, car il s'agit en particulier d'un idéal principal.

Une autre façon d'adapter aux corps finis les résultats obtenus sur $\C$ est d'utiliser les modules de Tate, qui sont des espaces naturels attachés aux courbes elliptiques sur lesquels $\End(E)$ agit de manière agréable. C'est la méthode suivie par Waterhouse \cite{Waterhouse}.

\newpage

\section{Courbes modulaires}

Une fois que l'on sait que l'action du groupe de classes est simplement transitive, on cherche naturellement à la calculer. Par exemple, pour calculer l'action d'un idéal de norme $\ell$, un nombre premier, il s'agit de déterminer une isogénie de degré $\ell$ au départ d'une courbe donnée.

On en arrive alors rapidement à la notion de courbe modulaire. Ces courbes paramétrisent la donnée d'une courbe elliptique munie d'une certaine structure: on dit qu'elles représentent un \emph{problème modulaire}. L'objectif de cette section est de montrer l'existence de ces courbes (ou au moins certaines d'entre elles) en toute généralité, c'est à dire en tant que schéma, et d'en calculer des équations. La référence principale est ici \cite{KaMa}. Comme précédemment, l'étude du cas complexe permet de motiver et d'introduire l'étude générale.

\subsection{Courbes modulaires sur $\C$}

Cherchons par exemple un espace classifiant les courbes elliptiques sur $\C$ munies d'un point de $N$-torsion ($N\geq 2$) à isomorphisme près. L'usage est de désigner ces objets par la lettre $Y$ complétée de quelques indices. Comme on l'a déjà mentionné, l'espace des courbes elliptiques complexe à isomorphisme près est le demi-plan de Poincaré quotienté par l'action du groupe $\Gamma(1) = \mathrm{SL}_2(\Z)$, ce que l'on note:
$$Y(1)_\C = \Gamma(1) \b \H.$$
La fonction modulaire $j$ donne un isomorphisme de cette surface de Riemann vers $\C$.

Soit $\tau\in \H$. La courbe elliptique $E_\tau = \C/(\Z+\Z\tau)$ a un point de $N$-torsion primitif naturel qui est simplement $\frac{1}{N}$. Cette donnée supplémentaire n'est pas invariante par tous les éléments de $\Gamma(1)$, seulement par ceux du sous-groupe
$$\Gamma_1(N) = \left\{\left(
\begin{matrix}
a & b \\
c & d
\end{matrix}
\right) \in \mathrm{SL}_2(\Z)\ :\ a = d = 1 \mod{N},\ c = 0\mod{N}\right\}.$$
On montre alors qu'en effet, une courbe elliptique complexe munie d'un point de $N$-torsion est isomorphe à un unique couple $\left(E_\tau, \frac{1}{N}\right)$ avec $\tau\in \Gamma_1(N) \b \H$. On note cela
$$Y_1(N)_\C = \Gamma_1(N) \b \H.$$

Qu'en est-t-il si l'on relâche la condition d'un point primitif de $N$-torsion, pour simplement demander un sous-groupe cyclique d'ordre $N$ ? Là encore, si $\tau\in \H$, $E_\tau$ a un groupe cyclique de cardinal $N$ naturel qui est simplement $\left\{\frac{a}{N},\ 0\leq a\leq N-1\right\}$. Cette donnée supplémentaire est invariante par un groupe plus grand que $\Gamma_1(N)$, à savoir
$$\Gamma_0(N) = \left\{\left(
\begin{matrix}
a & b \\
c & d
\end{matrix}
\right) \in \mathrm{SL}_2(\Z)\ :\ c = 0\mod{N}\right\}.$$
De même que précédemment, étant donné une courbe elliptique et un sous-groupe cyclique de cardinal $N$, il existe une unique $E_\tau$ munie de cette donnée supplémentaire qui lui soit isomorphe modulo $\Gamma_0(N)$, autrement dit
$$Y_0(N)_\C = \Gamma_0(N) \b \H.$$
Remarquons que ce problème modulaire est aussi celui de l'isogénie cyclique de degré $N$: la donnée d'un sous-groupe cyclique de $E$ de cardinal $N$ est équivalente à celle d'une isogénie cyclique partant de $E$ de degré $N$. Le terme \emph{cyclique} signifie ici que le noyau est cyclique.

\v

On peut compactifier les courbes $\Gamma \b\H$ pour tout sous-groupe de congruence $\Gamma$ de $\Gamma(1)$ en y ajoutant un nombre fini de points de $\P^1(\Q)$, appelés \emph{pointes}. On obtient alors des surfaces de Riemann compactes, qui sont donc des courbes algébriques, et que l'on notera avec la lettre $X$: ainsi $X(1)_\C = \P^1(\C)$. Pour étudier ces courbes, une direction naturelle est de se demander quelles sont leurs fonctions et formes différentielles, ne serait-ce que pour en trouver une équation. Bien sûr, ce sont simplement des fonctions modulaires de différents poids pour $\Gamma$, et l'on voit un premier lien entre les courbes elliptiques et les formes modulaires.

\subsection{Calcul d'équations pour les courbes modulaires}

Le but étant d'utiliser une courbe modulaire pour calculer des isogénies entre courbes elliptiques sur un corps fini, il est important de disposer d'équations pour ce genre de courbes. Mis à part le cas de $X(1)$, qui est isomorphe à la droite projective par le $j$-invariant, il est difficile de déterminer de telles équations à la main.

On aimerait travailler avec des équations planes, quitte à accepter des équations définissant des courbes singulières lorsque l'on est capable de contrôler les points singuliers. Une application birationnelle vers une courbe plane (projective) étant simplement la donnée de deux fonctions sur la courbe qui engendrent son corps de fonctions, on recherche tout d'abord ces dernières (par exemple sous forme de $q$-expansion) puis une équation polynomiale les reliant. La variable $q$ est définie par $q = e^{2i\pi\tau}$ lorsque $\tau$ parcourt $\H$, et la plupart des fonctions modulaires admettent une expression sous forme de série de Laurent en $q$.
\v

\textbf{Le cas $Y_0(N)$.} Ce cas est particulier, d'une part car il s'agit de la principale courbe modulaire utilisée dans ce document, et d'autre part car on peut en trouver une équation \og sans calcul \fg. On a vu que l'on peut reformuler le problème modulaire associé à $Y_0$ sous la forme du problème de $N$-isogénie cyclique $E\to E'$. On peut alors définir une application informelle
$$(E\to E') \longmapsto (E, E'),$$
ce qui suggère que l'application
$$\begin{aligned}
Y_0(N)_\C &\vers \C^2 \\
 \tau &\longmapsto (j(\tau), j(N\tau))
\end{aligned}$$
est bien définie, ce que l'on peut vérifier à la main.

Les fonctions modulaires $j(\tau)$ et $j(N\tau)$ sont simples, et on peut trouver une équation avec un papier et un crayon de la manière suivante. L'idée est que $\tau\mapsto N\tau$ est l'action d'une certaine matrice, et que l'on veut symétriser cela pour obtenir un polynôme à coefficients dans $\Z$. On définit
$$C(N)=\left\{ 
\left(
\begin{matrix}
a & b \\
0 & d 
\end{matrix}
\right)
\in \M_2(\Z)\ :\ ad=N,\ a>0,\ 0\leq b<d,\ \mathrm{pgcd}(a,b,d)=1\right\}.$$
et les $(\sigma^{-1}\Gamma(1)\sigma)\cap \Gamma(1)$ pour $\sigma\in C(N)$ sont exactement les classes à droite de $\Gamma_0(N)$ dans $\Gamma(1)$.

On montre alors qu'il existe un unique  polynôme $\Phi_N \in \Z[X,Y]$, appelé $N$-ième \emph{polynôme modulaire}, tel que
$$\forall \tau\in\H,\ \Phi_N(X,j(\tau))=\prod_{\sigma\in C(N)} (X-j(\sigma\tau)).$$
Ce polynôme est de plus symétrique en $X, Y$ et irréductible dans $\Z[X, Y]$. Ainsi l'image de l'application $Y_0(N)_\C\vers \C^2$ est exactement le lieu des zéros de ce polynôme $\Phi_N$. Les propriétés de $\Phi_N$ sont démontrées par exemple dans \cite{Sil2}.

Cependant, on ne peut pas vraiment dire que $\Phi_N=0$ définit une équation de la courbe modulaire $Y_0(N)_\C$, car le morphisme $Y_0(N)\vers \C^2$ n'est pas injectif. En effet, il existe des couples de réseaux $(\Lambda, \Lambda')$ tels qu'il existe plusieurs inclusions $\Lambda\vers\Lambda'$ de conoyau $\Z/N\Z$. Ces points singuliers comprennent notamment des points de $Y(1)$ ayant multiplication complexe par $\Z[i]$ ou $\Z[e^{\frac{2i\pi}{3}}]$, ce qui implique l'existence d'automorphismes non triviaux, mais ce ne sont pas les seuls. Pour contrôler ces points singuliers, on peut faire la remarque suivante. Si $\Lambda, \Lambda'$ sont comme ci-dessus, alors on dispose de deux isogénies cycliques de degré $N$:
$$\phi_1, \phi_2\de \C/\Lambda \vers \C/\Lambda'.$$
La composée $\widehat{\phi_1} \phi_2$ est un endomorphisme de la courbe elliptique $\C/\Lambda$ qui n'est pas scalaire, et est de degré $N^2$. Cela montre que $\C/\Lambda$ (et également $\C/\Lambda'$) sont des courbes à multiplication complexe, et dont le discriminant est borné par $N^2$. Déterminer plus précisément les points doubles de l'équation $\Phi_N = 0$ semble moins évident, mais cette borne grossière suffira dans ce mémoire.

Ainsi la courbe $\Phi_N=0$ détermine une courbe plane singulière qui, sur $\C$, est birationnelle à $Y_0(N)$. En dehors des points doubles, on dispose donc tout de même d'une équation qui permet, étant donnée une courbe, de déterminer les courbes que l'on peut atteindre par une isogénie cyclique de degré $N$.

\v

\textbf{Calcul du polynôme modulaire.} D'une certaine manière, la définition précédente du polynôme modulaire est inutile, car elle ne donne pas de moyen direct de le calculer (on apprend tout de même son degré, qui est $N+1$, et qu'il est symétrique à coefficients dans $\Z$). Une méthode consiste à utiliser les $q$-expansions, avec $q = e^{2i\pi\tau}$, des fonctions modulaires $j(\tau)$ et $j(N\tau)$, que l'on obtient à partir de celles des fonctions de Weierstrass:
$$\begin{aligned}
j(\tau) &= \frac{1}{q} + 744 + 196884 q + 21493760 q^2 + 864299970 q^3 + \ldots\\
j(N\tau) &= \frac{1}{q^N} + 744 + 196884 q^N + 21493760 q^{2N} + 864299970 q^{3N} + \ldots
\end{aligned}$$
On peut calculer les premières puissances de ces fonctions modulaires, puis rechercher une équation polynomiale (symétrique) les reliant en faisant de l'algèbre linéaire. On trouve
$$\begin{aligned}
\Phi_2(X, Y) &=  X^3 + Y^3 - X^2 Y^2 + 1488(X^2Y + X Y^2)  - 162000(X^2 + Y^2) \\
&\quad + 40773375 XY  + 8748000000(X + Y) - 157464000000000\\
\Phi_3(X, Y) &= X^4 + Y^4 - X^3 Y^3 + 2232(X^3 Y^2 + X^2 Y^3) - 1069956(X^3 Y + X Y^3)\\
&\quad + 2587918086 X^2 Y^2  + 36864000(X^3 + Y^3) + 8900222976000(X^2 Y + X Y^2)  \\
&\quad + 452984832000000(X^2 + Y^2) - 770845966336000000 XY \\
&\quad + 1855425871872000000000(X + Y)\\
\Phi_5(X, Y) &= X^6 + Y^6 -X^5 Y^5 + 3720(X^5 Y^4 + X^4 Y^5) - 4550940(X^5 Y^3 + X^3 Y^5) \\
&\quad + 1665999364600 X^4 Y^4 + 2028551200(X^5 Y^2 + X^2 Y^5) \\
&\quad + 107878928185336800(X^4 Y^3 + X^3 Y^4) - 246683410950(X^5 Y + X Y^5) \\
&\quad + 383083609779811215375(X^4 Y^2 + X^2 Y^4) - 441206965512914835246100 X^3 Y^3 \\
&\quad + 1963211489280(X^5 + Y^5) + 128541798906828816384000(X^4 Y + X Y^4) \\
&\quad + 26898488858380731577417728000(X^3 Y^2 + X^2 Y^3) \\
&\quad +  1284733132841424456253440(X^4 + Y^4) \\
&\quad - 192457934618928299655108231168000(X^3 Y + X Y^3) \\
&\quad + 5110941777552418083110765199360000 X^2 Y^2 \\
&\quad + 280244777828439527804321565297868800(X^3 + Y^3) \\
&\quad + 36554736583949629295706472332656640000(X^2 Y + X Y^2) \\
&\quad + 6692500042627997708487149415015068467200( X^2 + Y^2) \\
&\quad - 264073457076620596259715790247978782949376 XY \\
&\quad + 53274330803424425450420160273356509151232000(X + Y) \\
&\quad + 141359947154721358697753474691071362751004672000
\end{aligned}$$
et l'on s'arrête ici de peur de rapidement doubler la longueur de ce mémoire: le stockage du polynôme modulaire $\Phi_N$ en mémoire a un coût d'environ $\Theta(N^3)$ bits. On peut en fait calculer ces polynômes plus efficacement, comme Elkies l'indique dans \cite{Elkies}.
Des bases de données de polynômes modulaires se trouvent dans le logiciel de calcul formel Sage \cite{Sage} (qui contient également une riche base de formes modulaires) jusqu'à environ $N = 100$, ou sur la page web d'Andrew Sutherland (\url{math.mit.edu/~drew}), par exemple.

En pratique, lorsque la taille de ces polynômes devient problématique, on peut utiliser d'autres fonctions modulaires qui donnent de meilleurs résultats que $j(\tau)$ et $j(N\tau)$, par exemple la fonction de Weber. On ne gagne asymptotiquement qu'un facteur constant: une équation modulaire de grand niveau reste de toute manière volumineuse.

\v

\textbf{Le cas général.} Cette méthode est adaptable directement au cas d'autres courbes modulaires, dès que l'on est capable d'exhiber des fonctions modulaires du bon niveau qui engendrent son corps des fonctions. On peut ensuite faire de l'algèbre linéaire sur les $q$-expansions si l'on ne connaît pas a priori d'équation polynomiale satisfaite par ces fonctions modulaires. On verra dans ce mémoire deux exemples de cette méthode, pour les courbes $X_0(30)$ et $X_1(17)$.

Nous avons travaillé jusqu'à présent sur $\C$, mais il est crucial de savoir que ces équations, en particulier celles de $X_0(N)$, restent valables sur un corps quelconque (en particulier sur les corps finis).

Les premiers résultats concernant la réduction de courbes modulaires en caractéristique positive sont dus à Igusa. Il montre que les courbes modulaires classiques de niveau $N$ sont définies sur $\Q$, et ont bonne réduction et une interprétation modulaire en tous les premiers $p$ ne divisant pas $N$. C'est le contenu du théorème suivant:

\begin{thm}[Courbes modulaires en toute caractéristique] Soit $N\geq 3$ un entier.

\begin{itemize}

\item[(i)] La courbe modulaire $Y_0(N)_\C$ peut être définie sur $\Q$.
\item[(ii)] Pour tout $p \nmid N$, celle-ci a bonne réduction en $p$, c'est à dire qu'elle se réduit en une courbe lisse définie sur $\F_p$.
\item[(iii)] Pour tout corps $k$ dans lequel $N$ est inversible, la courbe $Y_0(N)_k$ a une interprétation modulaire. Les $k$-points de $Y_0(N)_k$ sont en bijection avec les classes d'isomorphisme (sur $\bar{k}$) d'isogénies cycliques de degré $N$ entre deux courbes elliptiques sur $k$.
\item[(iv)] Pour tout corps $k$, il existe un morphisme de schémas
$$Y_0(N) \vers \A^1\times \A^1$$
donné par le $j$-invariant comme dans le cas complexe. Son image est exactement le lieu des zéros du polynôme modulaire $\Phi_N$.

\end{itemize}

\end{thm}

Un résultat similaire est valable pour les autres types de courbes modulaires que l'on a vus ci-dessus. L'objectif du reste de cette section est de donner une idée d'une preuve de ce résultat.

\subsection{Le formalisme des problèmes modulaires}

Il est impossible de définir ici tout le vocabulaire de la théorie des schémas utilisé dans cette partie. Deux références pour cela: le livre de Hartshorne \cite{Hart} et surtout le Stacks project \cite{Stack}.

Soit $S$ un schéma. On appelle \emph{courbe elliptique} sur $S$ un $S$-schéma
$$\E \overset{\pi}{\vers} S$$
qui est une courbe propre et lisse, dont les fibres géométriques sont connexes de genre 1, munie d'une section notée $\O_\E\de S\vers \E$.
 
La condition sur les fibres signifie que pour tout point géométrique $x$ de $S$ (c'est à dire tout morphisme $x\de\Spec k\vers S$ où $k$ est un corps algébriquement clos), la \emph{fibre} de $\E$ en $x$, c'est à dire le $k$-schéma
$$\E \times_S \Spec k\ \vers\ \Spec k$$
est connexe et de genre 1. C'est de plus une courbe propre et lisse (ces propriétés sont préservées par changement de base) munie de la section $0 \times x$, donc une courbe elliptique sur le corps $k$ au sens usuel. Ainsi, on peut voir une courbe elliptique relative $\E/S$ comme une famille de courbes elliptiques \og compatibles \fg\ sur des corps variables, les points géométriques de $S$. Lorsque $S = \Spec k$, on retrouve la définition précédente d'une courbe elliptique sur un corps.

On peut montrer que $\E$ admet une structure naturelle de $S$-schéma en groupes commutatif qui étend la structure de schéma en groupe définie sur les fibres géométriques : c'est le théorème d'Abel.
\v

On note $(\Ell)$ la catégorie dont les objets sont les courbes elliptiques relatives
$$\E \overset{\pi}{\vers} S$$
pour des schémas arbitraires $S$, et dont les morphismes sont les carrés commutatifs cartésiens
$$
\shorthandoff{;:!?}
\xymatrix @!=8mm {
\E' \ar[d]^{\pi'} \ar[r]  & \E \ar[d]^{\pi} \\
 S' \ar[r] & S
}
$$
c'est à dire ceux pour lesquels on a $\E' \simeq \E \times_S S'$.

On appelle \emph{problème modulaire} pour les courbes elliptiques un foncteur contravariant $\Pr$ de $(\Ell)$ vers la catégorie des ensembles. Un élément de $\Pr(\E/S)$ est appelé \emph{structure de niveau} $\Pr$ sur $\E/S$. Lorsque l'on se restreint à prendre $S$ dans les $R$-schémas et aux morphismes de $R$-schémas, on obtient la sous-catégorie $(\Ell/R)$. Un \emph{problème modulaire sur $R$} est un foncteur contravariant comme ci-dessus défini sur cette sous-catégorie.

Les exemples étudiés dans le cas complexe sont des problèmes modulaires: par exemple, le problème modulaire correspondant à $Y_0(N)$ associe à une courbe elliptique $E$ l'ensemble des sous-groupes finis plats, localement libres de rang $N$ et cycliques dans $E[N]$.

\vspace{5mm}

On dit qu'un problème modulaire $\Pr$ est \emph{relativement représentable} si pour toute courbe elliptique $\E/S$, le foncteur (contravariant) des $S$-schémas vers les ensembles
$$T \ \longmapsto\ \Pr(\E\times_S T / T)$$
est représentable par un $S$-schéma noté $\Pr_{\E/S}$. Cela signifie que l'on a un isomorphisme fonctoriel en $T$ :
$$\Pr(\E\times_S T / T) \simeq \Pr_{\E/S}(T) = \Hom_{S-\text{sch.}}(T, \Pr_{\E/S}).$$
Si $\Pr$ est relativement représentable, on dit qu'il vérifie une certaine propriété si tous les morphismes structurels $\Pr_{E/S}\vers S$ ont cette propriété (par exemple, être étale, fini, surjectif, affine). Pour résumer, un problème modulaire $\Pr$ est relativement représentable si pour toute courbe elliptique fixée, les structures de niveau $\Pr$ de cette courbe sur une base variable $T$ correspondent aux $T$-points d'un certain schéma.

On dit que $\Pr$ est \emph{représentable} s'il est représentable en tant que foncteur sur $(\Ell$), c'est à dire qu'il existe une courbe elliptique relative
$$E \overset{p}\vers \M(\Pr)$$
telle qu'on ait un isomorphisme fonctoriel dans $(\Ell)$ :
$$\Pr(\E/S) \simeq \Hom_{(\Ell)}(\E/S, E/\M(\Pr)).$$
Dans ce cas, on dispose d'un élément universel $a\in \Pr(E/\M(\Pr))$ induit par l'identité de $E/\M(\Pr)$.

On remarque que si le problème $\Pr$ est représentable par $E\vers \M(\Pr)$, alors $\M(\Pr)$ représente le foncteur suivant dans la catégorie des schémas :
$$S\ \longmapsto\ \left.
\begin{cases}\ \text{Classes\ d'isomorphisme\ des\ paires}\ (\E/S,\alpha)\ \text{où}\ \\
 \ \E/S\ \text{est\ une\ courbe\ elliptique\ relative\ et}\ \alpha\in \Pr(\E/S)
\end{cases} \right\}.$$
C'est de cette manière que l'on utilisera la représentabilité dans ce document.

\v

On donne deux exemples de problèmes modulaires représentables : d'une part la donnée de deux points de 2-torsion qui engendrent $E[2]$, d'autre part la donnée de deux points de 3-torsion qui engendrent $E[3]$, que l'on notera respectivement $\Qr_2$ et $\Qr_3$ (il faut en fait ajouter une condition de compatibilité avec l'accouplement de Weil). Le premier est représentable sur l'anneau $\Z[1/2,\lambda][1/\lambda(\lambda -1)]$ par la courbe de Legendre:
$$ E\de y^2 = x (x-1) (x-\lambda),$$
et le second par la courbe
$$E\de y^2 + a_1 x y + a_3 y = x^3$$
où $a_1 = 3C - 1$ et $a_3 = -3 C^2 - B - 3 BC$, munie des deux points
$$P_3 = (0,0), \quad Q_3 = (C, B+C)$$
sur l'anneau
$$\frac{\Z[1/3, B, C][1/(a_1^3 - 27 a_3)a_3 C]}{B^3 = (B+C)^3}.$$

En dehors de ces exemples très simples, on ne sait pas donner explicitement un schéma représentant un problème modulaire. Il faut invoquer une construction plus abstraite: on peut montrer d'abord la relative représentabilité, puis invoquer un résultat de \emph{rigidité} pour en déduire la représentabilité. Cet argument est détaillé dans la partie suivante, mais ne donne pas accès à des équations du schéma construit.

\subsection{Représentabilité et rigidité}

Une première obstruction à la représentabilité des problèmes modulaires est la notion de \emph{rigidité}. On dit que $\Pr$ est \emph{rigide} si pour toute structure $\alpha$ de niveau $\Pr$ sur $\E/S$, la paire $(\E/S, \alpha)$ n'admet pas d'automorphismes non triviaux. Une conséquence immédiate est formelle de la définition de représentabilité est la suivante: tout problème modulaire représentable est rigide. Un fait important est que cette obstruction est essentiellement la seule.

\begin{thm}[Représentabilité des problèmes rigides]
Tout problème modulaire affine relativement représentable et rigide est représentable.
\end{thm}

C'est la stratégie que nous suivons pour obtenir la représentabilité des problèmes modulaires donnés plus haut. Montrer leur relative représentabilité est plus facile, car on travaille sur une courbe fixée: cette question est étudiée dans \cite{KaMa}, (3.7).

Soit $\Pr$ un problème modulaire affine, rigide, et relativement représentable. Le premier ingrédient de la preuve est le suivant: si $\Qr$ est représentable et $\Pr$ est relativement représentable, on montre que le problème modulaire \og simultané\fg\ $\Pr\times\Qr$ est représentable par le schéma 
$$\M(\Pr, \Qr) = \Pr_{E/\M(\Qr)}.$$

L'idée de la preuve du théorème de représentabilité est alors de choisir un problème représentable $\Qr$ particulièrement simple, regarder le schéma $\M(\Pr, \Qr)$, et le \emph{quotienter} par quelque chose pour faire disparaître le problème $\Qr$. En guise de $\Qr$, on utilise le problème modulaire de Legendre $\Qr_2$. Il n'est défini que sur $\Z[1/2]$, donc on regarde la même construction avec le problème $\Qr_3$, qui lui est défini sur $\Z[1/3]$. Autrement dit, on obtient une solution

\begin{itemize}
\item[$\bullet$] Pour tous les schémas où 2 est inversible, d'une part;
\item[$\bullet$] Pour tous les schémas où 3 est inversible, d'autre part.
\end{itemize}
Moralement, la rigidité de $\Pr$ impose à ces deux solutions de coïncider lorsqu'elles le doivent, et on peut donc \og recoller\fg tout cela en une solution pour tous les schémas.

Donnons maitenant le résultat qui fait marcher cette construction.

\begin{lem}
Soit $N\geq 1$ un entier, $G$ un groupe fini, et $\Qr$ un problème modulaire affine et relativement représentable. On suppose de plus que $\Qr$ est représentable par un schéma affine sur $\Z[1/N]$ et que $G$ agit sur $\Qr$ de telle sorte que pour toute courbe elliptique
$$\E\vers S\vers \Spec \Z[1/N],$$
le $S$-schéma $\Qr_{\E/S}$ soit un $G$-torseur fini étale. Alors $\Pr$ est représenté par le $\Z[1/N]$-schéma affine
$$\M(\Pr,\Qr)/G.$$
\end{lem}

\begin{proof}[Démonstration du lemme.] On a dit que le problème simultané $(\Pr, \Qr)$ est représentable par
$$\M(\Pr, \Qr) = \Pr_{E/\M(\Qr)}.$$
Comme $\Pr$ est supposé affine, ce schéma est affine sur $\M(\Qr)$ qui est lui-même affine, donc est affine. De plus, le groupe $G$ agit sur $\M(\Pr, \Qr)$, car celui-ci représente le foncteur $\Pr\times\Qr$, et on peut faire agir fonctoriellement $G$ sur la composante $\Qr$.

\newcommand{\univ}{\mathrm{univ}}

Considérons alors la courbe elliptique universelle pour $\Pr\times\Qr$ :
$$E \vers \M(\Pr,\Qr)$$
munie de $(\alpha_{\univ}, \beta_{\univ}) \in (\Pr\times\Qr)(E/\M(\Pr,\Qr)).$ La preuve se déroule alors en trois temps :
\begin{itemize}
\item[•] Montrer l'existence du quotient $\M(\Pr,\Qr)/G$ ;
\item[•] Descendre $(E, \alpha_{\univ})$ à $\M(\Pr,\Qr)/G$, c'est à dire remplir le diagramme cartésien
$$
\shorthandoff{;:!?}
\xymatrix {
(E,\alpha_{\univ}) \ar[d] \ar@{.>}[r] &\ \textbf{?}\ \ar@{.>}[d] \\
 \M(\Pr,\Qr) \ar[r] & \M(\Pr,\Qr)/G
}
$$
\item[•] Montrer que la courbe descendue représente bien le problème modulaire $\Pr$ sur $\Z[1/N]$.
\end{itemize}

Pour le premier point, il suffit de montrer que le groupe fini $G$ agit librement sur le schéma affine $\M(\Pr,\Qr)$; alors on sait que le quotient existe par le théorème de Grothendieck, et que la projection
$$\M(\Pr,\Qr)\vers \M(\Pr,\Qr)/G$$
est un $G$-torseur fini étale (et surjectif, donc fidèlement plat). Pour cela, il suffit de montrer que l'action est \emph{universellement libre}. Soit $T$ un $\Z[1/N]$-schéma. On a vu que $\M(\Pr, \Qr)(T)$ est l'ensemble des courbes elliptiques sur $T$ munies d'une $(\Pr,\Qr)$-structure à isomorphisme près. Si un élément $g\in G$ laisse un élément invariant, cela signifie que l'on a un certain isomorphisme
$$(\E/T, \alpha, \beta) \simeq (\E/T, \alpha, g\beta).$$
Par rigidité de $\Pr$, cet isomorphisme doit être l'identité. Ainsi on a $g\beta = \beta$, et par hypothèse $\Qr_{E/T}$ est un $G$-torseur, ce qui implique que $g$ est lui-même trivial. Cela conclut la preuve que l'action de $G$ sur $\M(\Pr,\Qr)$ est libre.

Pour descendre $E, \alpha_{\univ}$ par ce quotient, on est dans le cadre général de la descente fidèlement plate. On cherche une \emph{donnée de descente}, c'est à dire un ensemble d'isomorphismes pour $g\in G$:
$$\theta(g)\ :\ g^*(E, \alpha_{\univ}) \overset{\sim}{\vers} (E, \alpha_{\univ})$$
qui forment un 1-cocycle (i.e. ils sont compatibles à la composition). On peut obtenir une telle donnée de la façon suivante. Pour tout $g\in G$, $(E, \alpha_{\univ}, g\beta_{\univ})$ est une courbe elliptique sur $\M(\Pr, \Qr)$ munie d'une $(\Pr,\Qr)$-structure, donc est classifiée par un unique morphisme
$$\theta(g) \de \M(\Pr,\Qr)\vers \M(\Pr,\Qr)$$
qui donne un isomorphisme au-dessus de $\M(\Pr,\Qr)$:
$$g^*(E, \alpha_\univ, \beta_\univ) \overset{\sim}{\vers} (E, \alpha_\univ, g\beta_\univ)$$
et il suffit d'oublier $\beta_\univ$. La compatibilité avec la composition provient directement de la rigidité de $\Pr$.

Ainsi $E$ et $\alpha_\univ$ descendent, et on a rempli le carré cartésien précédent avec un objet $(E_0,\alpha_{\univ, 0})$ sur $\M(\Pr,\Qr)/G$.

Il reste maintenant à montrer que l'on obtient ainsi un représentant du problème $\Pr$. 
Pour cette vérification finale, on renvoie à la preuve originale dans \cite{KaMa}.

\end{proof}

\v

L'objectif de cette section est atteint si le lecteur s'est fait une idée du type de constructions impliquées dans la construction de schémas \og non triviaux\fg. On peut tirer de cette preuve d'autres renseignements.

\begin{cor}
Sous les hypothèses du théorème précédent, si de plus $\Pr$ est étale, alors $\M(\Pr)$ est une courbe affine lisse sur $\Z$.
\end{cor}
En effet, on a obtenu dans la preuve une construction du schéma affine $\M(\Pr)$ par recollement de quotients finis de $\M(\Pr, \Qr_2)$ et $\M(\Pr, \Qr_3)$. Or, les deux morphismes
$$\M(\Pr, \Qr_2)\vers \M(\Qr_2),\quad \M(\Pr, \Qr_3)\vers \M(\Qr_3)$$
sont étales, car $\Pr$ est supposé étale. Il suffit donc de montrer que $\M(\Qr_2)$ et $\M(\Qr_3)$ sont des courbes lisses sur $\Z$, car cette propriété est préservée par les morphismes étales et les quotients par un groupe fini agissant librement. Cette dernière assertion se vérifie directement sur les anneaux en question.

\begin{cor}
Pour tout entier $N\geq 3$, le problème modulaire du point d'ordre $N$ (correspondant à $Y_1(N)$) est représentable par une courbe affine lisse sur $\Z[1/N]$, notée $Y_1(N)$.
\end{cor}

En effet, on a vu que ce problème modulaire est relativement représentable, fini et étale donc en particulier affine. D'autre part, il est rigide: on peut donc appliquer le résultat précédent.

On a donc vu comment étendre les résultats obtenus sur $\C$ à n'importe quel corps pour le problème rigide associé à $Y_1(N)$. En revanche cela ne s'applique pas à $Y_0(N)$, qui est non rigide car l'action de $[-1]$ est triviale: ce problème modulaire n'est pas représentable. 

\v

Lorsque l'on dispose d'un problème modulaire relativement représentable et affine sur $(\Ell/R)$, on peut tout de même lui associer un schéma appelé \emph{schéma de modules grossier}, noté $M(\Pr)$. Ce $R$-schéma est simplement $\M(\Pr)$ si $\Pr$ est représentable ; si $\Pr$ n'est pas représentable, il fait office de meilleur remplaçant. Le point qui nous intéresse est que ce schéma, même s'il ne représente pas $\Pr$, a une interprétation en termes modulaires lorsque l'on se restreint à considérer les points définis sur un corps:

\begin{prop}
Soit $k$ un corps et $\bar{k}$ une clôture algébrique de $k$. On a l'identification
$$ M(\Pr)(k) = 
\left.
\begin{cases}
\ \text{Courbes\ elliptiques}\ \E/k \ \text{munies\ d'une\ structure} \\
\ \text{de\ niveau}\ \Pr,\ \text{à}\ \bar{k}\text{-isomorphisme\ près}.
\end{cases}
\right\}$$
\end{prop}

Par exemple, on a vu que le $j$-invariant classifie les courbes elliptiques sur un corps $k$ à $\bar{k}$-isomorphisme près. Un $j$-invariant peut prendre toutes les valeurs de $k$, et on peut en fait montrer que le schéma de modules grossier associé au problème modulaire \og courbe elliptique à isomorphisme près\fg est bien la droite affine $\A^1$.

On peut noter $Y_0(N)$ le schéma sur $\Z[1/N]$ qui vérifie la propriété ci-dessus. Une fois que l'on connaît l'existence de ce schéma, il n'y a plus vraiment de difficulté pour démontrer le théorème d'Igusa tel qu'énoncé plus haut. 

\v

L'objectif de cette partie est donc atteint: on dispose d'une équation, le polynôme modulaire,: dont les solutions sont les $j$-invariants de courbes liées par des isogénies cycliques de degré donné. En pratique, d'autres équations de courbes obtenues sur $\C$ restent valides sur $\F_p$: on en verra des exemples dans la suite de ce mémoire. Ces outils permettent d'effectuer les calculs nécessaires au protocole de Couveignes--Rostovtsev--Stolbunov, que l'on étudie maintenant.

\newpage

\section{Cryptosystème et algorithmes}

\subsection{\'Echanges de clés}

Deux objectifs fondamentaux en cryptographie sont d'une part assurer une communication \emph{privée} entre deux parties, et d'autre part assurer l'\emph{authentification} d'une entité à une autre. Ces conditions peuvent être obtenues à l'aide de la \emph{cryptographie symétrique}, où les deux entités utilisent une clé secrète partagée afin de crypter les données. Cependant, dans de nombreuses applications, une telle clé secrète n'est souvent pas connue au préalable. Il faut alors faire appel à la \emph{cryptographie à clefs publiques}, ou cryptographie asymétrique.

Dans ce document, on s'intéresse à des protocoles d'\emph{échanges de clés}, dont le but est de parvenir à la formation d'un secret commun à deux parties pouvant communiquer via un canal ouvert qu'un adversaire peut observer. Il est donc impossible d'envoyer ce secret \og en clair\fg. Une fois l'échange de clés effectué, les protagonistes utilisent ce secret commun pour en dériver une clé (à l'aide d'une fonction de hachage, par exemple) puis s'en servir au sein d'un protocole de cryptographie symétrique de leur choix afin d'échanger des données. Bien que cette seconde partie soit importante en pratique, on ne discutera dans ce mémoire que d'échange de clés proprement dits.

Un exemple de protocole d'échange de clés est le protocole Diffie--Hellman, proposé en 1976. Dans un groupe $G$, typiquement le groupe multiplicatif d'un corps fini, on dispose d'un élément $g$ d'ordre $N$, tous ces éléments étant publics. Alice et Bob souhaitent partager un secret $K$. Pour cela, 
\begin{itemize}
\item[•]Alice choisit $a\in (\Z/N\Z)^\times$, calcule $c_1 = g^a$ et l'envoie à Bob.
\item[•]Bob choisit $b\in (\Z/N\Z)^\times$, calcule $c_2 = g^b$ et l'envoie à Alice.
\item[•]Recevant $c_1$, Bob calcule $K = c_1^b.$
\item[•]Recevant $c_2$, Alice calcule $K = c_2^a.$
\end{itemize}
Les deux protagonistes connaissent alors $K = g^{ab}$, mais un adversaire n'a vu que $g,\, g^a$ et $g^b.$ Lorsqu'un adversaire \og passif\fg\ observe cet échange et souhaite trouver $K$, il résout donc le problème suivant.

\v
\noindent \textbf{Problème computationnel de Diffie--Hellman (CDH).} \'Etant donné $(g, g^a, g^b)$, calculer $g^{ab}$.
\v

Lorsqu'il souhaite deviner si une clé devinée est la bonne, il résout un analogue \og décisionnel\fg\ de ce problème:

\v
\noindent \textbf{Problème décisionnel de Diffie--Hellman (DDH).} \'Etant donné $(g, g^a, g^b, g^c)$, décider si $g^c = g^{ab}$.

\v
Ainsi, sans vouloir rentrer dans les détails, la \emph{sécurité} du protocole de Diffie--Hellman repose sur le postulat que ces problèmes sont difficiles (lorsque $N$ est premier, notamment) lorsque $a$ et $b$ sont choisis uniformément au hasard dans $(\Z/N\Z)^\times$.

\v

Dans la situation ci-dessus, on remarque que le groupe abélien $G = (\Z/N\Z)^\times$ agit sur l'ensemble $X = \{g^k, k\in (\Z/N\Z)^\times\}$ de façon simplement transitive, $X$ étant muni d'un point fixé $g$. On peut alors reformuler le protocole de Diffie et Hellman de manière plus générale, suivant une idée de Couveignes \cite{Couv}. Soit $G$ un groupe abélien agissant de façon simplement transitive sur un ensemble $X$ muni d'un point fixé $x_0$. Le protocole s'écrit de la façon suivante:
\begin{itemize}
\item[•]Alice choisit $a\in G$, calcule $c_1 = a\cdot x_0$ et l'envoie à Bob.
\item[•]Bob choisit $b\in G$, calcule $c_2 = b\cdot x_0$ et l'envoie à Alice.
\item[•]Recevant $c_1$, Bob calcule $K = b\cdot c_1.$
\item[•]Recevant $c_2$, Alice calcule $K = a \cdot c_2.$
\end{itemize}
Alice et Bob ont calculé $K = (ab)\cdot x_0 = (ba) \cdot x_0$, puisque le groupe $G$ est supposé abélien. Comme précédemment, la sécurité de ce protocole repose sur la difficulté du problème suivant: 

\v
\noindent\textbf{Problème 1.} \'Etant donné $x_0, a\cdot x_0, b\cdot x_0\in X$, calculer $ab\cdot x_0$.
\v

À son tour, celui-ci repose sur le problème d'\og inversion\fg\ de l'action de groupe:

\v
\noindent\textbf{Problème 2.} \'Etant donné $x_0, x_1 \in X$, calculer $a\in G$ tel que $a\cdot x_0 = x_1$.
\v

Couveignes appelle \emph{espace homogène difficile} la donné de $G, X, x_0$ tels que
\begin{itemize}
\item[•]L'action est simplement transitive;
\item[•]L'action est facile à calculer;
\item[•]Le problème 2 ci-dessus est \emph{difficile}.
\end{itemize} 
Ici, on dit qu'un problème est \emph{difficile} s'il n'existe pas d'algorithme polynomial en $n$ donnant une solution avec une probabilité non négligeable pour une entrée de taille $n$.

Couveignes propose d'utiliser l'action d'un groupe de classes sur un ensemble de courbes elliptiques sur un corps fini, décrite au début de ce document. Il s'agit du cryptosystème décrit dans l'article de Rostovtsev et Stolbunov \cite{RoSt}. L'objectif du reste de ce mémoire est d'étudier ce protocole, les différents algorithmes utilisés et son applicabilité en pratique, et d'en rechercher quelques améliorations.

\subsection{Représentation des objets mathématiques}

On fixe un corps fini $k$ de cardinal $p$, un nombre premier, et un entier $t\leq 2\sqrt{p}$. Ce choix d'un corps premier n'est pas motivé par des raisons théoriques, mais uniquement pratiques: c'est le corps dans lequel les opérations sont le plus efficaces, à taille comparable. On note $K$ le corps quadratique
$$K = \Q[\pi] / (\pi^2 - t\pi + q),$$
on se donne une courbe elliptique $E_0/k$ dont la trace du Frobenius est égale à $t$, et on note $\O$ l'ordre de $K$ qui est son anneau d'endomorphismes. On détaille maintentant la façon dont sont représentés divers objets mathématiques \og en machine\fg.

\v

\emph{Courbes elliptiques.} Une courbe elliptique est d'abord une courbe algébrique plane, et on peut donc la représenter par une de ses équations. On utilise deux formes de représentations de courbes elliptiques : la forme de Weierstrass réduite
$$\E\de y^2 = x^3 + Ax + B$$
et la forme dite de Montgomery
$$\E\de B y^2 = x^3 + A x^2 + x.$$
Notons que toute courbe elliptique sur $k$ n'admet pas nécessairement de représentation sous forme de Montgomery : en effet, une telle courbe admet en particulier un point rationnel de 2-torsion qui est $(0,0)$, et ce n'est pas le cas de toutes les courbes.
On s'autorise également à représenter une courbe elliptique à isomorphisme près par son $j$-invariant, qui est un élément de $k$. Ce faisant, on perd de l'information par rapport aux équations ci-dessus, car on ne peut plus distinguer une courbe de sa tordue.

Sur un corps fini, une courbe a en général une unique tordue, qui lui est isomorphe sur une extension quadratique de $k$ mais pas sur $k$ lui-même. Tordre une courbe a pour effet de changer sa trace de signe. Dans le modèle de Montgomery, la tordue a une forme agréable : il s'agit de la courbe
$$E'\de \varepsilon B y^2 = x^3 + Ax^2 + x$$
où $\varepsilon$ est un non-résidu quadratique dans $k$. Ajouter deux points sur une courbe elliptique sous forme de Montgomery est également plus efficace que sur une courbe sous forme de Weierstrass, ce qui est l'intérêt principal de ces courbes.
\v

\emph{Sous-groupes finis d'une courbe elliptique et isogénies.} Si $S$ est un sous-schéma en groupe fini (et plat) de $E$ d'ordre $\ell$, où $p\neq \ell$ et $\ell$ est impair, on représente le sous-groupe $S$ de la façon suivante. $S$ est étale donc déterminé par ses $\bar{k}$-points, et est stable par l'automorphisme $[-1]$ de $E$.
En écrivant une équation pour $E$, il existe donc un polynôme $K_S\in k[X]$ de degré $\frac{\ell-1}{2}$ dont les racines sont exactement les coordonnées $x$ des points de $S$. On représente alors $S$ par ce polynôme. Remarquons que cette représentation n'est univoque que si l'on se restreint aux sous-groupes étales.

Comme on l'a vu précédemment, une isogénie séparable $\phi$ est déterminée par sa courbe de départ et son noyau. Si elle est de degré impair, on peut donc la représenter par un polynôme comme ci-dessus, que l'on appelle \emph{polynôme de noyau}, noté $K_\phi$. On représentera donc uniquement des isogénies séparables de cette manière.

Dans le cas de sous-groupes d'ordre pair et d'isogénies de degré pair, on peut adopter la même stratégie et définir un polynôme en $x$ qui s'annule sur les points du sous-groupe. L'inconvénient en est que les zéros de ce polynôme $K(x)$ n'auront pas partout la même multiplicité, ce qui complique les formules. On considèrera uniquement des isogénies de degré impair dans la suite de ce mémoire. 
\v

\emph{Idéaux.} On dit qu'un nombre premier $\ell$ est \emph{Elkies} si le polynôme $X^2 - tX + q$ est scindé à racines simples modulo $\ell$, c'est à dire si son discriminant est un carré non nul modulo $\ell$. Si $\ell$ est un nombre premier d'Elkies, alors $\ell$ se scinde dans $\O$ sous la forme $(\ell) =~\frak l \bar{\frak l},$
et ces deux idéaux sont exactement les idéaux de norme $\ell$ dans $\O$. Pour les distinguer, on peut utiliser la remarque suivante. Si $E$ est une courbe elliptique ayant les bons paramètres, on sait que $E[\ell](\bar{k})$ est un $\O/\ell\O$-module de rang 1. D'autre part, on a un isomorphisme canonique par le théorème chinois :
$$\O/\ell\O \simeq \O/\frak l \O \times \O/\bar{\frak l} \O.$$
L'endomorphisme de Frobenius $\pi$, vu comme endomorphisme du $\F_\ell$-espace vectoriel $E[\ell](\bar{k})$ de dimension 2, admet donc deux valeurs propres (nécessairement distinctes car $\ell$ est d'Elkies) qui sont exactement les éléments $\pi\mod \frak l$ et $\pi\mod \bar{\frak l}$ : on les appelle \emph{valeurs propres du Frobenius} modulo $\ell$. On représentera alors un idéal de norme $\ell$ par le couple $(\ell, v)$ où $v$ est la valeur propre du Frobenius associée. D'un point de vue pratique, les valeurs propres du Frobenius sont les racines dans $\F_\ell$ du polynôme $X^2 - tX + q$ ; on reviendra plus tard sur le calcul de $t$.

Une isogénie de degré $\ell$ définie sur $k$ partant de $E$, à isomorphisme près, est simplement un sous-groupe cyclique de $E(\bar{k})$ de cardinal $\ell$ stable par Galois (i.e. par le Frobenius). On vient de voir que l'action du Frobenius sur $E[\ell](\bar{k})$ est diagonale, avec des valeurs propres distinctes. On en déduit que les seules isogénies rationnelles de degré $\ell$ partant de $E$ sont celles qui proviennent de l'action du groupe de classes.
\v

\emph{Idéaux fractionnaires et groupe de classes.} On utilisera toujours des idéaux fractionnaires écrits sous la forme $\prod {\frak l}_i^{r_i}$, où les ${\frak l}_i$ sont des idéaux dont la norme est un petit nombre premier d'Elkies. Cela implique, en particulier, que l'on ne travaille qu'avec un sous-groupe du groupe des idéaux fractionnaires.

On choisit de donner un élément de $\Cl(\O)$ par un de ses représentants, que l'on écrit comme ci-dessus. Notons que l'on pourrait toujours atteindre le groupe de classes en entier même en se restreignant à un sous-groupe des idéaux fractionnaires. On décomposera donc toujours l'action d'un élément du groupe de classes en actions successives d'idéaux \og simples\fg, c'est à dire de la forme $(\ell, v)$ où $\ell$ est un petit nombre premier d'Elkies.

On peut donner à cette décomposition un aspect plus visuel à l'aide du concept de \emph{graphe d'isogénies}. Les sommets de ce graphe sont des courbes elliptiques sur $k$ à isomorphisme près (donc des $j$-invariants), et l'on relie ces sommets par une arête labellée par un premier $\ell$ s'il existe une isogénie de degré $\ell$ les reliant. Comme on l'a dit, l'action du groupe de classes par isogénies est simplement transitive, et on obtient donc simplement un graphe de Cayley pour le groupe de classes relativement aux \og générateurs\fg\ $(\ell, v)$ pour de petits $\ell$. Cela explique sa régularité:

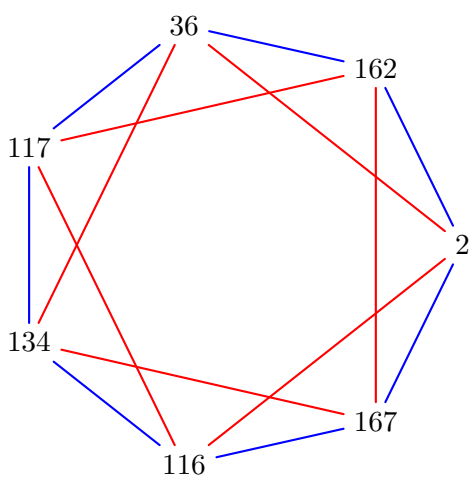
\begin{center}
\begin{tikzpicture}[line cap=round,line join=round, x=3cm,y=3cm]

\draw[]
 (1, 0) node(1) {2}
 (0.62, 0.78) node(2) {162}
 (-0.22,0.97) node(3) {36}
 (-0.9,0.43) node(4) {117}
 (-0.9,-0.43) node(5) {134}
 (-0.22,-0.97) node(6) {116}
 (0.62,-0.78) node (7) {167};
\draw[]
 (1) edge[thick, blue] (2)
 (2) edge[thick, blue] (3)
 (3) edge[thick, blue] (4)
 (4) edge[thick, blue] (5)
 (5) edge[thick, blue] (6)
 (6) edge[thick, blue] (7)
 (7) edge[thick, blue] (1);
\draw[]
 (1) edge[thick, red] (3)
 (3) edge[thick, red] (5)
 (5) edge[thick, red] (7)
 (7) edge[thick, red] (2)
 (2) edge[thick, red] (4)
 (4) edge[thick, red] (6)
 (6) edge[thick, red] (1);

\end{tikzpicture}
\end{center}

On montre ici un graphe d'isogénies sur $\F_{173}$; les sommets sont des $j$-invariants, les arêtes bleues représentent des isogénies de degré 3, et les arêtes rouges des isogénies de degré 7. La décomposition précédente revient simplement à découper un parcours dans ce graphe en plusieurs pas. Le graphe en couverture est un graphe d'isogénies sur $\F_{503}$, avec le degré 3 en bleu, 11 en rouge et 13 en vert. Dans ce graphe, pour un degré donné, les deux valeurs propres du Frobenius \og décrivent\fg\ les deux sens de parcours possibles.

\v 

Ce découpage reflète la manière dont l'action du groupe de classes sera calculée. En effet, on ne connaît pas de moyen efficace de calculer directement la courbe image d'une isogénie de grand degré. Afin de calculer l'action d'\emph{un} élément du groupe de classes sur \emph{une} courbe, on effectuera donc toute une série de pas dans ce graphe d'isogénies. Pour fixer les idées, l'algorithme~\ref{alg:act} décrit la manière de calculer une action.

\v

\begin{algorithm}[H]
	\label{alg:act}
		\caption{\textsc{Action}: calcul de l'action d'un élément de $\Cl(\O)$ sur une courbe}
		
    \KwIn{$E$ courbe, $g\in \Cl(\O)$ représenté sous la forme $\prod_i {\frak l}_i ^{k_i}$, où ${\frak l}_i$ est l'idéal correspondant au couple $(\ell_i, v_i)$}
    
    \KwOut{$E'$ telle que $E \to E'$
    est une isogénie correspondant à l'action de $g$}
    \For{$i$}{
        \lFor{\(0 \le j < k_i\)}{%
            $E \gets $\textsc{Pas}$(E, \ell_i, v_i)$
        }
    }
    \Return{\(E\)}
\end{algorithm}

\v
Afin de compléter la description de l'échange de clés de Rostovtsev et Stolbunov, on précise que la \og génération aléatoire\fg\ d'un élément dans le groupe de classes se fait directement sous la forme d'un produit d'idéaux: il ne s'agit donc pas de la distribution uniforme. Plus précisément, on fixe un ensemble $L$ de petits premiers d'Elkies, et une borne $M_\ell$ pour tout $\ell\in L$. L'algorithme est alors le suivant:
\begin{itemize}
\item[•] Pour tout $\ell\in L$, choisir un entier $- M_\ell\leq n_\ell\leq M_\ell$ uniformément;
\item[•] Retourner l'élément $\prod_{\ell\in L} {\frak l}^{n_\ell}$ du groupe de classes, où $\frak l$ est un idéal de $\O$ de norme $\ell$ choisi arbitrairement (autrement dit, une valeur propre du Frobenius choisie arbitrairement).
\end{itemize}
On en déduit minoration du nombre de pas à effectuer au total: pour assurer que l'\og espace de clés\fg soit de taille $N$, il faut au moins que $\prod_{\ell\in L} (2M_\ell + 1)\geq N$. On y reviendra à la section 4.

\v

Dans le reste de cette section, on donne plusieurs manières d'effectuer un \textsc{Pas} dans ce graphe d'isogénies. Sauf précision contraire, on discute du coût des algorithmes en termes d'opérations dans $\F_p$.

\subsection{Calcul à l'aide d'un polynôme de division}

Si $\phi\de E\vers E'$ est une isogénie séparable de degré $\ell$, alors son noyau est constitué de points de $E$ d'ordre $\ell$. Le polynôme $K_\phi$ est donc un facteur de degré $\frac{\ell-1}{2}$ à coefficients dans $k$ du $\ell$-ième \emph{polynôme de division} de la courbe $E$, qui est simplement le polynôme de noyau de l'isogénie $[\ell]$ de degré $\ell^2$. Avec cette idée, on obtient l'algorithme suivant.

\v

\begin{algorithm}[H]
\caption{un pas dans le graphe d'isogénies à l'aide des polynômes de division}
\label{alg:div}
\KwIn{Une courbe elliptique $E\in \Ell_k(\O)$, et un idéal représenté par $(\ell, v)$}
\KwOut{Une courbe elliptique $E'$ telle que l'isogénie $E\vers E'$ corresponde à cet idéal}
$\psi_\ell \gets \textsc{PolynômeDeDivision}(E, \ell)$
\label{alg:div:poldiv}
\;
$F \gets \textsc{Factorisation}(\psi_\ell)$
\label{alg:div:fact}
\;
$K \gets \textsc{SousGroupe}(F, \frac{\ell-1}{2}, v)$
\label{alg:div:rec}
\;
\lElse{\Return{$\textsc{Quotient}(E, K)$}}

\end{algorithm}

\v

Seul le but de l'étape~\ref{alg:div:rec} nécessite peut-être une explication : il consiste à retrouver à partir des facteurs irréductibles de $\psi_\ell$ le polynôme de degré $\frac{\ell-1}{2}$ définissant le sous-groupe cycliques de $E$ qui est le sous-espace propre du Frobenius sur $E[\ell](\bar{k})$ associé à la valeur propre $v$. Cette étape est nécessaire car ces polynômes ne sont pas nécessairement irréductibles, et $\psi_\ell$ peut avoir beaucoup d'autres facteurs qui ne déterminent pas des sous-groupes mais seulement des sous-ensembles de $E[\ell](\bar{k})$.

Examinons maintenant chacune des étapes de l'algorithme~\ref{alg:div}, en commençant par le calcul des polynômes de division.

\begin{lem}[Polynômes de division]
Soit $E$ une courbe elliptique sous forme de Weierstrass réduite
$$y^2 = x^3 + Ax + B.$$
Il existe des polynômes dits \emph{de division} $\Psi_n \in \Z[x,y]/(y^2 - x^3 - Ax - B)$, $n\geq 1$ tels que pour tout $n\geq 1$, l'endomorphisme $[n]_E$ soit donné sur l'ouvert affine par les applications rationnelles
$$(x,y) \longmapsto \left(\frac{x\Psi_n^2 - \Psi_{n-1}\Psi_{n+1}}{\Psi_n^2},
\frac{\Psi_{n+2}\Psi_{n-1}^2 - \Psi_{n-2} \Psi_{n+1}^2}{4y \Psi_n^3}\right),$$
en posant $\Psi_0 = 0,\ \Psi_{-1} = -1,$ et tels que l'on ait :
$$\begin{aligned}
 \Psi_1&= 1,\\
 \Psi_2&= 2y,\\
 \Psi_3&= 3x^4 + 6Ax^2 + 12Bx - A^2, \\
 \forall n\geq 2,\ 2y\Psi_{2n} &= \Psi_n(\Psi_{n+2}\Psi_{n-1}^2 - \Psi_{n-2} \Psi_{n+1}^2), \\
\forall n\geq 2 ,\ \Psi_{2n+1} &= \Psi_{n+2}\Psi_n^3 - \Psi_{n+1}^3\Psi_{n-1}.
\end{aligned}$$
\end{lem}

Lorsque $n$ est impair, on montre aisément par récurrence que $\Psi_n$ peut être vu comme un polynôme en $x$ uniquement et de degré $\frac{n^2-1}{2}$. Si de plus $p\nmid n$, on sait que $\#E[n](\bar{k})=n^2$, donc $\Psi_n$ est un polynôme séparable dont les racines sont exactement les coordonnées $x$ des points de $n$-torsion de la courbe. 
On utilise ces relations de récurrence, obtenues à partir des formules d'additions de points sur la courbe, pour écrire l'algorithme~{\sc PolynômeDeDivision} dans l'étape~\ref{alg:div:poldiv}.
\v

Pour l'étape \ref{alg:div:fact}, {\sc Factorisation}, on utilise l'algorithme de Cantor--Zassenhaus \cite{vzGG}. La première étape est de séparer les facteurs irréductibles d'un polynôme $P$ par degré, ce que l'on fait en prenant le pcgd avec les polynômes $X^p - X$, puis $X^{p^2} - X$, etc. que l'on calcule directement modulo $P$ par mises au carré répétées. Ensuite, on tente de factoriser le reste obtenu $R$ en prenant un polynôme au hasard $a$ et en calculant $\mathrm{pgcd}(a^{\frac{p^r - 1}{2}} + 1, R)$. L'idée de l'algorithme est que si $R$ est un produit de facteurs irréductibles de degré $r$, alors $k[X]/R$ est un produit de corps de cardinal $p^r$, et le calcul de cette puissance donne aléatoirement 1 ou -1 dans chacun de ces corps.
\v

Dans l'étape \ref{alg:div:rec}, il s'agit de retrouver le polynôme à coefficients dans $k$ définissant un certain sous-groupe cyclique d'ordre $\ell$. Pour cette opération, on peut utiliser l'astuce suivante. Comme on connaît la valeur propre du Frobenius $v$ sur le sous-groupe recherché, on peut calculer l'action du Frobenius sur chacun des facteurs irréductibles de $\psi_\ell$ et rassembler les morceaux. Pour un facteur irréductible $Q$, on calcule $[v](X, Y)$ pour le point tautologique $(X, Y)$ sur l'anneau $k[X, Y]/(Q,\ Y^2 - X^3 - AX - B).$ Si ce point est égal à $(X^p, Y^p)$, on sait que l'on a trouvé un facteur du polynôme de noyau $K$.
\v
%

Enfin, on souhaite calculer une équation de la courbe image connaissant $E$ et le noyau $K$ de l'isogénie. C'est là qu'intervient la preuve \og directe\fg\ de l'existence du quotient évoquée plus tôt dans ce document.

\begin{lem}[Formules de Vélu]
Soit $E$ une courbe elliptique donnée par une équation sous forme de Weierstrass
$$y^2 + a_1 x y + a_3 y = x^3 + a_2 x^2 + a_4 x + a_6,$$
et soit $K = \sum_{i=0}^n (-1)^{n-i} \sigma_i X^i$ un polynôme définissant un sous-groupe de $E$ cyclique d'ordre $2n+1$. Notons $b_2, b_4, b_6, b_8$ les $b$-invariants de l'équation de $E$, et 
$$ t = 6(\sigma_1^2 - 2 \sigma_2) + b_2 \sigma_1 + n b_4,$$
$$ w = 10 (\sigma_1^3 - 3 \sigma_1 \sigma_2 + 3 \sigma_3) + 2  b_2 (\sigma_1^2 - 2 \sigma_2) + 3 b_4 \sigma_1 + n b_6.$$
Alors, en notant $E'$ la courbe donnée par les invariants
$$ a_1' = a_1,\ a_2' = a_2,\ a_3' = a_3,\ a_4' = a_4 - 5 t,\ a_6' = a_6 - b_2 t - 7 w,$$
il existe une isogénie séparable $E\vers E'$ de noyau $K$.
\end{lem}

Les $b$-invariants sont des quantités qui apparaissent lors de la réduction de $E$ en forme de Weierstrass courte; ce sont des expressions polynomiales simples en les $a_i$. L'idée de la preuve est la suivante.
Soit $G$ le sous-groupe étale de $E$ défini par $K$. On définit deux fonctions rationnelles sur $E$:
$$\begin{aligned}
x_G(P) = x(P) + \sum_{Q\in G\backslash \{0\}} x(P+Q) - x(Q),\\
y_G(P) = y(P) + \sum_{Q\in G\backslash \{0\}} y(P+Q) - y(Q).
\end{aligned}$$
Ces fonctions rationnelles sont bien définies et invariantes par $G$ : elles définissent donc des fonctions sur la courbe $E/G$. On montre alors qu'elles satisfont l'égalité donnée par l'équation de Weierstrass ci-dessus, et qu'elles engendrent le corps des fonctions de la courbe quotient.
Une discussion plus détaillée sur ces formules se trouve dans \cite{Kohel}, notamment leur extension aux isogénies cycliques de degré pair.

\v
On dispose depuis récemment de meilleures formules dans le cas de la forme de Montgomery, données dans \cite{VeluMontgomery}: si $P$ est un point d'ordre $2n+1$ sur une courbe de Montgomery d'équation $By^2 = x^3 + Ax^2 + x$, alors il existe une isogénie normalisée dont le noyau est le sous-groupe engendré par $P$ vers la courbe d'équation
$$B'y^2 = x^3 + A'x^2 + x$$
avec $A' = (6\sigma + A)\pi^2,\ B' = B\pi^2$, où
$$
\pi = \prod_{i = 1}^n x_{iP}, \quad
\sigma = \sum_{i = 1}^n \left(\frac{1}{x_{iP}} - x_{iP}\right).
$$
Comme ci-dessus, on peut réexprimer $\pi$ et $\sigma$ en termes des coefficients du polynôme $K$.

Sans utiliser ces formules, une autre méthode est disponible : étant donné une courbe sous forme de Montgomery, on peut calculer une équation de Weierstrass de cette courbe, et utiliser les formules de Vélu pour cette courbe. En se souvenant du point de 2-torsion, on peut ensuite retrouver une équation de Montgomery pour la courbe quotient sans extraction de racine carrée.

\v
D'un point de vue pratique, factoriser entièrement un polynôme de degré $\frac{\ell^2 - 1}{2}$ sur $k$ lorsque $p$ est grand et $\ell$ n'est pas très petit, est une opération trop coûteuse. L'étape de reconstruction du sous-groupe n'est pas non plus gratuite.
Cela pousse à préférer une autre méthode à l'aide d'une équation modulaire, ce que l'on décrit maintenant.

\subsection{Calcul à l'aide d'une équation modulaire}

On a vu que la courbe modulaire $Y_0(\ell)_k$ paramétrise les isogénies cycliques de degré $\ell$ définies sur $k$. Le polynôme modulaire de degré $\ell$ relie les $j$-invariants de ces courbes ; on peut montrer de plus que les points singuliers sur $k$ de l'application
$$Y_0(\ell) \vers \A^1 \times \A^1,$$
dont l'image est définie par le polynôme modulaire, sont nécessairement des réductions modulo $p$ de points singuliers complexes, donc peu nombreux (et dont l'ordre a un petit discriminant). En utilisant ces outils, on obtient l'algorithme suivant pour un calcul d'isogénie, qui est valide lorsque le discriminant de la courbe est supérieur à $\ell^2 + 1$, selon la borne grossière obtenue précédemment.

\v

\begin{algorithm}[H]
\caption{un pas dans le graphe d'isogénies en utilisant une équation modulaire}
\label{alg:mod}
\KwIn{Une courbe elliptique $E\in \Ell_k(\O)$, et un idéal représenté par $(\ell, v)$}
\KwOut{Une courbe elliptique $E'$ telle que l'isogénie $E\vers E'$ corresponde à cet idéal}
$\Phi_\ell(X, Y) \gets \textsc{PolynômeModulaire}(\ell)$
\label{alg:mod:polmod}
\;
$(j_1, j_2) \gets \textsc{Racines}(\Phi_\ell(j(E), Y), k)$
\label{alg:mod:roots}
\;
$K_1 \gets \textsc{Noyau}(E, \ell, j_1)$
\label{alg:mod:ker}
\;
\lIf{$\textsc{ValeurPropre}(K_1,v)$}{\Return{$\textsc{Courbe}(j_1)$}}
\label{alg:mod:check}
\lElse{\Return{$\textsc{Courbe}(j_2)$}}

\end{algorithm}

\v

De la même façon, on étudie les étapes de l'algorithme l'une après l'autre. On a vu une manière de calculer le polynôme modulaire, mais le plus rapide reste d'utiliser une base de données préfabriquée. Pour le calcul des racines, on utilise à nouveau l'algorithme de Cantor--Zassenhaus, mais le polynôme à factoriser est cette fois de degré seulement $\ell + 1$. Enfin, l'algorithme {\sc ValeurPropre} est analogue à la reconstruction du sous-groupe dans l'algorithme 1, mais le calcul du Frobenius est effectué modulo un unique polynôme de degré $\frac{\ell-1}{2}$. 

\v
Examinons maintenant le calcul du noyau. Soit $\phi\de E\vers E'$ une isogénie. On a une application $k$-linéaire \og pullback\fg, des formes différentielles sans pôle de $E'$ vers celles de $E$:
$$\phi^*\de H^0(E', \Omega^1)\vers H^0(E, \Omega^1).$$
Rappelons que $E$ et $E'$ sont de genre 1, donc ces $k$-espaces vectoriels sont de dimension 1. On sait que cette application est nulle si et seulement si $\phi$ est non séparable. Si l'on fixe une équation de Weierstrass pour $E$ et $E'$, on dispose de 1-formes privilégiées et donc d'un isomorphisme (non canonique)
$$\Hom(H^0(E', \Omega^1), H^0(E, \Omega^1))\simeq k.$$
On dit alors que $\phi$ est \emph{normalisée} si $\phi^* = 1$ dans cette identification. Lorsque $\phi$ est normalisée, la condition de pullback donne une équation différentielle satisfaite par $\phi$ qui est donnée par la proposition qui suit, et qui est à la base de l'algorithme de calcul de noyau. Cette idée est due à Elkies et Stark.

\begin{prop}
Soit $\phi\de E\vers E'$ une isogénie normalisée de degré $\ell$ entre courbes sous forme de Weierstrass telles que $a_1 = a_3 = a_1' = a_3' = 0$ (ainsi les équations sont de la forme $y^2 = f(x)$). Notons $\phi_x(x,y),\ \phi_y(x,y)$ les applications rationnelles définissant $\phi$ sur l'ouvert affine ; alors $\phi_x$ est une fonction paire, donc peut s'écrire sous forme irréductible
$$\phi_x(x, y) = \frac{N(x)}{D(x)}$$
où $N$ et $D$ sont des polynômes de degré au plus $\ell$. Notons
$$G(X) = a_6 X^3 + a_4 X^2 + a_2 X + 1,\ H(X) =a_6' X^3 + a_4' X^2 + a_2' X + 1.$$
Alors la fraction rationnelle 
$$T(X) = \frac{D(1/X)}{N(1/X)}$$
est solution de l'équation différentielle
$$\frac{T}{X} H(T) - G(X) T'^2 = 0.$$
\end{prop}

\begin{proof}
Comme $\phi$ est normalisée, on sait que $\phi_y (x, y) = y \phi_x'(x).$ En écrivant 
$$\phi_x(x) = \frac{N(x)}{D(x)} = \frac{1}{T(1/x)},$$
on a donc
$$\phi_y(x, y) = \frac{y}{x^2} \frac{T'(1/x)}{T(1/x)^2}$$
et
$$ \left(\frac{y}{x^2} T'(1/x)\right)^2 = \left(\frac{1}{T^3(1/x)} + \frac{a_2'}{T^2(1/x)} + \frac{a_4'}{T(1/x)} + a_6'\right)T(1/x)^4$$
ce qui donne, vu l'équation de $E$ :
$$ \frac{1}{x} G(1/x) T'^2(1/x) =T(1/x) H(T(1/x))$$
et on obtient l'équation annoncée après un changement de variable.
\end{proof}
\v

Remarquons que dans cette description, lorsque $\ell$ est impair, le polynôme $D$ est précisément le carré du polynôme de noyau $K_\phi$. Précisons maintenant comment se déroule le calcul du noyau.

\v

\begin{algorithm}[H]
\caption{{\sc Noyau} : Calcul du noyau de l'isogénie}
\label{alg:ker}
\KwIn{Une courbe elliptique $E\in \Ell_k(\O)$, un nombre premier $\ell\neq p$, un élément $j'\in k$ qui est le $j$-invariant d'une courbe reliée à $E$ par une isogénie cyclique de degré $\ell$}
\KwOut{Le noyau de cette isogénie}
$E' \gets \textsc{EquationNormalisée}(E, \ell, j')$
\label{alg:ker:eq}
\;
$T \gets \textsc{SolutionEquaDiff}(E, E')$
\label{alg:ker:newt}
\;
$D \gets \textsc{Dénominateur}(1/T(1/X))$
\label{alg:ker:bm}
\;
$K\gets \textsc{RacineCarrée}(D)$
\label{alg:ker:sqrt}
\;
\Return{$K$}

\end{algorithm}

\v

A nouveau, prenons ces étapes dans l'ordre, à commencer par le calcul d'une équation normalisée.

\begin{lem}
Avec les notations de l'algorithme, soit $\Phi_ \ell(X, Y)$ le polynôme modulaire de degré $\ell$. On suppose que $E$ est donnée sous forme réduite
$$y^2 = x^3 + A x + B.$$
On définit
$$\lambda = \frac{-18}{\ell}\cdot\frac{B}{A}\cdot\frac{\frac{\partial \Phi_\ell}{\partial X} (j(E), j')}{\frac{\partial \Phi_\ell}{\partial Y} (j(E), j')} \cdot j(E) $$
Alors, si $E'$ désigne la courbe d'équation
$$y^2 = x^3 + A' x + B' $$
avec
$$A' = \frac{-\lambda^2}{48 \ell^4 j' (j' - 1728)},\ B' = \frac{-\lambda^3}{864 \ell^6 j'^2(j'-1728)},$$
on a $j(E') = j'$ et il existe une isogénie normalisée $E\vers E'$ de degré $\ell$.
\end{lem}

Dans le cas complexe, on peut en effet supposer que la courbe $E$ est écrite sous la forme $E_\tau = \C/\Lambda_\tau$, et que l'isogénie de degré $\ell$ s'écrit le la manière suivante:
$$\C/(\Z\oplus \Z\tau) \vers \C/(\Z \oplus \Z\frac{\tau}{\ell}).$$
 On écrit alors des égalités entre séries formelles :
$$A(\tau) = \frac{- E_4}{48}, \quad B(\tau) = \frac{E_6}{864}$$
où $E_4$ et $E_6$ sont les séries d'Eisenstein classiques, et on peut tout écrire explicitement en fonction de $q$ comme dans \cite{Schoof}. Cela donne les formules ci-dessus.

Dans le cas d'un corps fini, on peut utiliser le théorème de Deuring, avec le même fait admis que dans la première section. On sait alors que l'isogénie définie sur $\F_p$ est la réduction de l'isogénie définie sur les complexes, ce qui justifie de réduire dans $\F_p$ les équations donnant $A'$ et $B'$. Il peut arriver que la dérivée $\frac{\partial \Phi_\ell}{\partial Y}$ s'annule modulo $p$; dans ce cas, l'algorithme de calcul du noyau échoue. Dans l'algorithme \ref{alg:mod}, on tente alors de calculer le noyau de l'autre isogénie. Il faudrait être bien malchanceux pour que le second calcul échoue également.

\v

Pour résoudre l'équation différentielle jusqu'à une précision donnée, on utilise une itération de Newton. On sait déjà en effet que la solution $T$ de l'équation n'a pas de terme constant, et cette itération permet de déterminer une solution dans $k[[X]]/(X^{2^{i+1}})$ à partir d'une solution dans $k[[X]]/(X^{2^i})$. On commence avec la solution $0$ dans $k[[X]]/(X)$. Concrètement, le calcul prend la forme suivante : avec les notations de l'algorithme~\ref{alg:ker}, on définit
$$T_0 = 0, \quad \forall i\geq 0,\ T_{i+1} = T_i + T_i' \sqrt{G} \sqrt{X} \int \frac{T_i(X)}{2\sqrt{X}}.$$
Alors pour tout $i\geq 0$, $T_i$ est une solution de l'équation différentielle modulo $X^{2^i}$.

Les algorithmes de calcul de noyau proposés à l'origine par Stark et Elkies calculent les coefficients de $T$ de proche en proche, ce qui entraînait un coût quadratique en $\ell$. L'idée d'utiliser une itération de Newton pour résoudre cette équation est due à Bostan, Morain, Salvy et Schost \cite{BMSS}. Le coût de cet algorithme devient quasi-linéaire, c'est à dire quasi-optimal, en $\ell$. Remarquons toutefois que le calcul d'une équation normalisée demande encore de manipuler le polynôme modulaire en deux variables, qui est donné par $\frac{(\ell + 1)^2}{2}$ coefficients dans $\F_p$.

\v

 Pour récupérer le dénominateur, un simple algorithme d'Euclide étendu permet de récupérer les polynômes $N$ et $D$ lorsque l'on connaît les coefficients de la série formelle $\frac{N}{D}$ jusqu'au degré $2\ell + 1$. En effet, si l'on connaît un polynôme $U$ tel que $U = \frac{N}{D} \mod X^{2\ell + 1}$, alors il existe un polynôme $P$ tel que $D U = N + P X^{2\ell+1}$, ce que l'on peut réécrire
 $$D U - P X^{2\ell + 1} = N.$$
On lance alors l'algorithme d'Euclide étendu avec $U$ et $X^{2\ell + 1}$, que l'on arrête dès que les coefficients $N$ et $D$ sont tous deux de degré inférieur à $\ell$, ce qui survient nécessairement au cours du calcul.

Enfin, pour le calcul de racine carrée, on utilise la remarque suivante : si $D = K^2$ où $K$ est un polynôme séparable, alors $K$ et $\mathrm{pgcd}(D, D')$ sont associés. En effet, on a $D' = 2 K K'$ et pgcd$(K, K')$ = 1 puisque $K$ est séparable.
\v

Le coût principal de l'algorithme~\ref{alg:mod} est partagé entre le calcul des racines et l'algorithme {\sc ValeurPropre}, qui impliquent tous deux de calculer une puissance $p$-ième modulo un polynôme de degré de l'ordre de $\ell$. C'est une amélioration importante d'un point de vue pratique par rapport au $\ell^2$ précédent. Cependant, on peut proposer encore mieux dans certains cas particuliers: utiliser cette dernière méthode dans le cadre du cryptosystème de Couveignes--Rostovtsev--Stolbunov est une contribution originale de ce mémoire.

\subsection{Calcul à l'aide de torsion rationnelle}

Pour certaines valeurs propres spéciales du Frobenius, on peut proposer une autre méthode qui évite complètement le calcul d'une puissance $p$-ième. Par exemple, si 1 est valeur propre modulo $\ell$, cela signifie qu'il existe un sous-groupe cyclique d'ordre $\ell$ constitué de points rationnels sur $k$ (et un seul, car les valeurs propres sont distinctes). On obtient alors l'algorithme \ref{alg:tors}.

\v

\begin{algorithm}[H]
\caption{un pas dans le graphe d'isogénies à l'aide de points de torsion rationnels}
\label{alg:tors}
\KwIn{Une courbe elliptique $E\in \Ell_k(\O)$, et un idéal représenté par $(\ell, v)$ tel que $v = 1$}
\KwOut{Une courbe elliptique $E'$ telle que l'isogénie $E\vers E'$ corresponde à cet idéal}
$P \gets \textsc{PointDeTorsion}(E, \ell)$
\label{alg:tors:tors}
\;
$K \gets \textsc{SousGroupeEngendré}(E, P)$
\label{alg:tors:sg}
\;
$E' \gets \textsc{Quotient}(E, K)$
\label{alg:tors:quo}
\;
\Return{$E'$}

\end{algorithm}

\v

Remarquons qu'en dehors de la recherche du point de torsion, cette méthode est très peu coûteuse. Le coût de la construction du sous-groupe est linéaire en $\ell$, mais indépendant de $p$, et le quotient final est simplement un appel aux formules de Vélu, ce qui est quasiment instantané. La recherche d'un point de torsion rationnel ne pose pas non plus beaucoup de problèmes:

\v

\begin{algorithm}[H]
\caption{{\sc PointDeTorsion} : calcul d'un point de torsion rationnel}
\label{alg:ptors}
\KwIn{Une courbe elliptique $E\in \Ell_k(\O)$ de trace $t$, et un nombre premier d'Elkies $\ell$ tel que $E$ admet un point de $\ell$-torsion rationnel primitif}
\KwOut{Un point de $\ell$-torsion primitif}
$C \gets p + 1 - t$
\label{alg:ptors:card}
\;
\Repeat{$P \neq 0$}{
$Q \gets \textsc{PointAuHasard}(E,k)$
\;
$P \gets \left[\frac{C}{\ell}\right]Q$
\label{alg:ptors:scal}
    }
\Return{$P$}

\end{algorithm}

\v

Nous reviendrons plus tard sur le calcul de la trace $t$. Cette quantité est la même pour toutes les courbes rencontrées au cours du calcul: en effet, selon un théorème de Tate, deux courbes sur $k$ reliées par une isogénie définie sur $k$ ont le même nombre de points sur $k$. La trace $t$ dépend donc uniquement de la courbe d'origine choisie, et n'est ensuite plus recalculée.

Le coût principal de la recherche d'un point de torsion est la multiplication scalaire à l'étape~\ref{alg:ptors:scal}. Pour calculer ce type de multiplications, on utilise un mécanisme de chaîne d'additions (\og double-and-add\fg\ par exemple, en utilisant l'expression du cofacteur en base binaire) qui utilise environ $\log(C)$ (c'est à dire $\log(p)$) additions sur la courbe. Cela représente un gain important par rapport aux algorithmes~\ref{alg:div} et~\ref{alg:mod}. 

\v

L'algorithme~\ref{alg:tors} n'est utilisable que pour les nombres premiers $\ell$ pour lesquels 1 est valeur propre. Intuitivement, il y a environ une chance sur $\ell - 1$ que cela survienne pour une courbe choisie \og au hasard \fg : c'est de plus en plus rare à mesure que $\ell$ augmente. On peut donner une variante lorsque les points de $\ell$-torsion ne sont pas nécessairement $k$-rationnels, mais seulement définis sur une extension de $k$ de petit degré. Il faut alors multiplier un point choisi au hasard par $\frac{C'}{\ell}$, où $C'$ est le cardinal de la courbe sur cette extension (qui peut se calculer directement à partir de $C$ par la théorie de Hasse). Lorsque l'extension est de degré $d$, $C'$ est de l'ordre de $p^d$. Associé au fait qu'il est plus coûteux de manipuler des points définis sur une extension, le coût de cet algorithme étendu devient quadratique en $d$. Dans le cas générique, on a $d = \ell - 1$, ce qui implique qu'il est plus intéressant d'utiliser l'algorithme~\ref{alg:mod}: l'algorithme~\ref{alg:tors} n'est donc intéressant que pour de petites valeurs de $d$. De plus, le fait que $\ell$ soit d'Elkies n'empêche plus les deux sous-groupes d'apparaître simultanément, ce qui ajoute une nouvelle restriction.

En pratique, il est rare qu'une courbe admette beaucoup de nombres premiers pour lesquels l'algorithme \ref{alg:tors} ou une de ses variantes peut être utilisé. Pour profiter pleinement de l'accélération permise par cette méthode, il faut donc trouver des courbes elliptiques exceptionnelles, qui ont plus de points de torsion rationnelle de petit ordre que la moyenne.

\v
D'autre part, l'algorithme \ref{alg:tors} permet de voyager dans le graphe d'isogénies uniquement dans la direction associée à la valeur propre 1, et non dans la direction de l'autre valeur propre. On peut proposer trois astuces:
\begin{itemize}
\item[•] Utiliser le modèle de Montgomery permet d'avoir des opérations plus efficaces, ce qui réduit le coût de la multiplication scalaire. On peut ainsi ajouter deux points génériques en 6 additions, 2 mises au carré et 4 multiplications dans $k$, contre 8 additions, 5 multiplications, une mise au carré et une division pour un modèle de Weierstrass \og classique\fg\ non optimisé.

\item[•] Un autre avantage du modèle de Montgomery est de pouvoir travailler uniquement sur la coordonnée $x$. Il devient alors très facile de passer d'une courbe à sa tordue. En pratique, cela revient à dire que l'on tire autant d'avantages d'une valeur propre égale à $-1$ qu'une valeur propre égale à 1.

\item[•] Enfin, une dernière amélioration est de forcer l'égalité $p = -1 \mod\ell$. De ce fait, si $v_1 = 1$ dans une direction, alors on aura $v_2 = -1$ dans l'autre, ce qui garantit une méthode efficace pour pouvoir faire un pas dans les deux sens.
\end{itemize}
\v

En utilisant une méthode générique, on ne peut pas espérer beaucoup mieux que de devoir chercher les racines d'un polynôme de degré $\ell + 1$ sur $\F_p$. En effet, il existe toujours $\ell + 1$ isogénies distinctes de degré $\ell$ au départ d'une courbe donnée qui sont définies sur $\bar{k}$: c'est le nombre de sous-groupes cycliques de taille $\ell$ dans $E(\bar{k})$. Utiliser des points de torsion rationnelle d'ordre $\ell$ permet de contourner ce calcul lorsqu'ils existent.

La recherche de courbes ayant beaucoup de tels points devient alors critique pour l'amélioration des performances. Il est également intéressant de disposer d'une courbe ayant beaucoup de nombres premiers d'Elkies, car on sera amené à manipuler des isogénies de degré moins élevé. Cependant, avant de pouvoir chercher des courbes intéressantes, différents paramètres restent à déterminer.

\newpage
\section{Recherche de paramètres}

Plusieurs paramètres restent à fixer dans le cryptosystème décrit à la section précédente. D'une part, la taille de $k$, qui conditionne la taille du graphe d'isogénies et donc celle du groupe de classes. D'autre part, il reste également à définir quels idéaux utiliser pour se déplacer dans le graphe d'isogénie, c'est à dire déterminer la liste $L$ et les bornes $M_\ell$ dans la méthode de génération aléatoire précédente pour le groupe de classes:
\begin{itemize}
\item[•] Pour tout $\ell\in L$, choisir un entier $- M_\ell\leq n_\ell\leq M_\ell$ uniformément;
\item[•] Retourner l'élément $\prod_{\ell\in L} {\frak l}^{n_\ell}$ du groupe de classes, où $\frak l$ est un idéal de $\O$ de norme $\ell$ choisi arbitrairement (autrement dit, une valeur propre du Frobenius choisie arbitrairement).
\end{itemize}

Ces paramètres, comme leurs analogues dans d'autres protocoles cryptographiques, doivent être choisis de telle sorte à assurer une bonne sécurité contre les meilleures attaques connues. Ici, on parle uniquement de sécurité contre un adversaire \emph{passif}, qui observe les communications entre les deux parties mais n'interfère pas avec le protocole: l'adversaire cherche donc à retrouver le secret à partir des éléments publics.

Déterminer ces paramètres permettra également de mesurer les véritables performances de ce protocole. En effet, cette mesure n'a de sens qu'à un niveau de sécurité donné: par exemple, un intérêt de beaucoup de protocoles utilisant des courbes elliptiques par rapport au RSA classique est de pouvoir utiliser des groupes plus petits pour atteindre le même niveau de sécurité.

\v
Pour quantifier cette sécurité, on parle de \emph{bits}. On dit qu'un protocole présente $n$ bits de sécurité si la meilleure attaque connue nécessite au moins $2^n$ opérations pour réussir avec une probabilité non négligeable. Même sans préciser ce que signifie \og opération \fg, c'est une façon agréable de quantifier le concept de sécurité. En informatique, on aime les puissances de 2, c'est pourquoi les niveaux de sécurités standard sont aujourd'hui 128, 192 ou 256 bits de sécurité (ce dernier étant tout de même un peu paranoïaque). L'objectif de cette partie est de déterminer des paramètres de notre cryptosystème et d'estimer ses performances réelles au niveau d'environ 128 bits de sécurité.

\subsection{L'attaque classique}

Ce que l'on appelle ici \emph{attaque} est donc un algorithme (probabiliste ou non) permettant de déterminer la clé privée à partir des données publiques. Rappelons les données du problème:
\begin{itemize}
\item[•] Les éléments publics sont des courbes elliptiques $E_0$, $E_a$ et $E_b$, ainsi que la liste $L$ des degrés d'isogénies utilisés et les bornes $M_\ell$ correspondantes.
\item[•] Les éléments secrets sont les éléments du groupe de classe $a, b$ tels que $E_a = a\cdot E_0$ et $E_b = b\cdot E_0$.
\end{itemize}

Une façon d'attaquer le protocole est de chercher à déterminer $a$ à partir des données publiques $E_0$ et $E_a$. Autrement dit, il s'agit de chercher une isogénie entre deux courbes données; en termes de graphe d'isogénies, il s'agit de retrouver un chemin dans un graphe reliant deux points donnés.

\v
Pour cela, on peut construire une attaque fondée sur le \og paradoxe des anniversaires\fg. Cette attaque a été proposée par Galbraith dans \cite{Galbraith}, qui est la référence principale de cette section et dont sont tirés tous les résultats cités ici. Le résultat principal en est le suivant:

\begin{thm}
Soient $E_1, E_2$ des courbes elliptiques ordinaires définies sur $\F_p$ telles que $\#E_1(\F_p) = \#E_2(\F_p)$. En admettant l'hypothèse de Riemann pour les corps quadratiques, et sous heuristiques, il existe un algorithme probabiliste construisant une isogénie rationnelle entre $E_1$ et $E_2$. Le coût de cet algorithme est dans le pire des cas $O(p^{3/2}\ln(p))$ en temps et $O(p \ln(p))$ en espace.
\end{thm}

Le principe de cet algorithme est de construire deux arbres dans le graphe d'isogénies enracinés l'un en $E_1$, l'autre en $E_2$. Pour construire ces arbres, on calcule l'action d'un groupe de classes avec les algorithmes de la section précédente. Lorsque deux branches entrent en collision, on dispose d'un chemin entre $E_1$ et $E_2$, ce qui survient avec une bonne probabilité une fois que le nombre de branches est de l'ordre de $\sqrt{h}$, où $h$ est le nombre de classes. L'hypothèse de Riemann intervient ici pour s'assurer de pouvoir parcourir tout le groupe de classes avec des idéaux dont la norme est inférieure à une borne explicite. Cet algorithme repose de plus sur une hypothèse heuristique, à savoir que les branches de chaque arbre ne se recouvrent pas trop: autrement dit, les produits de la forme
$$\prod_{i = 1}^n {\frak l}_i^{a_i}, \quad a_i\in \{-1,0,1\}$$
se répartissent de manière raisonnablement uniforme dans le groupe de classes pour une grande proportion de suites $({\frak l}_i)$.

Avant de construire ces arbres, toutefois, il faut s'assurer de disposer de deux courbes ayant le même anneau d'endomorphismes, et ce n'est pas nécessairement le cas des deux courbes d'origine. Une solution consiste à utiliser des isogénies dont le degré divise le conducteur pour se ramener tout d'abord à des courbes dont l'anneau d'endomorphismes est maximal: les propriétés de ces isogénies \og exceptionnelles \fg\ sont décrites par exemple dans \cite{Kohel}. Cette étape est la plus coûteuse dans le pire des cas.
\v

L'attaque de Galbraith peut être simplifiée dans notre cadre. On sait déjà en effet que les deux courbes ont le même anneau d'endomorphismes, et que seuls des premiers de la liste $L$ ont été utilisés pour voyager entre elles. Cela permet d'éviter de se ramener à l'anneau d'endomorphismes maximal, et on peut aussi profiter du fait que les idéaux dont les normes apparaissent dans $L$ n'engendrent pas nécessairement la totalité du groupe de classes. On peut donc reformuler l'attaque ainsi.

\begin{prop}
Avec les notations précédentes, il existe un algorithme probabiliste qui, étant donnés $E_0$, $E_a$ et $L$, construit une chaîne d'isogénies entre $E_0$ et $E_a$ dont les degrés sont choisis dans $L$. En notant $N$ le cardinal du sous-groupe du groupe de classes engendré par les idéaux dont la norme figure dans $L$, le coût de cet algorithme est en moyenne $\Theta(\sqrt{N}\, \ln(N) \ln(p)^5)$.
\end{prop}

Remarquons que l'hypothèse heuristique n'est plus nécessaire, car les bornes $M_\ell$ donnent précisément le nombre de branches à construire pour chaque premier. Le facteur $\ln(p)^5$ reflète le coût des calculs d'isogénies comme ils ont été décrits à la section précédente, en utilisant les polynômes modulaires.

\v
Cette attaque jouant le rôle de la meilleure connue, le choix des paramètres $p$ et $L$ doit être fait de telle sorte que la quantité $N$ de la proposition précédente soit de l'ordre de 250 bits, c'est à dire $2^{250}$.

\subsection{Choix des paramètres}

On a vu que la sécurité du cryptosystème dépend de manière critique du cardinal d'un certain sous-groupe du groupe de classes. Trois questions se posent alors:
\begin{itemize}
\item[•] Comment peut-on contrôler, en fonction de $p$, la taille du groupe de classes ? Plus précisément, connaissant $p$ et la trace $t$ du Frobenius la courbe $E_0$, peut-on contrôler le cardinal du groupe de classes des ordres de $K = \Q(\pi)/(\pi^2 - t\pi + p)$ ?
\item[•] Pour un ordre $\O$ donné, comment peut-on déterminer une borne $M$ telle que l'ensemble des idéaux inversibles de $\O$ de norme inférieure à $M$ engendre un grand sous-groupe du groupe de classes ?
\item[•] Une fois cette borne définie, peut-on donner pour chaque idéal ${\frak l}_i$ de norme $\ell_i\leq M$, une borne $M_i$ telle que les produits
$$\prod_{i} {\frak l}_i^{a_i}, \quad - M_i\leq a_i\leq M_i$$
se répartissent à peu près uniformément dans ce sous-groupe ? Plus généralement, peut-on donner un critère pour que des bornes $M_i$ soient convenables ?
\end{itemize}

Ces questions d'arithmétiques sont complexes, et souvent les seules réponses que l'on sait apporter sont heuristiques. Lorsque des théorèmes existent, il s'agit de théorie analytique des nombres, et leurs preuves sont hors de portée de ce mémoire.
\v

\textbf{Contrôle du nombre de classes.} Un premier élément est qu'au sein d'un même corps quadratique imaginaire, le nombre de classes augmente à mesure que le conducteur augmente. On peut en trouver une preuve chez Lang \cite{Lang}, th. 8.7.

\begin{prop}
Soit $\O$ un ordre de conducteur $c$ dans un corps quadratique imaginaire $K$. Notons $h_\O$ et $h_K$ le nombre de classes de $\O$ et $\O_K$ respectivement. Alors on a l'égalité
$$h_\O = h_K c\, [\O_K^*:\O^*] \prod_{p|c} \left(1 - \left(\frac{K}{p}\right)p^{-1}\right)$$
où l'on note $\left(\frac{K}{p}\right) = -1, 0, 1$ selon si $p$ est inerte, ramifié ou scindé dans $K$, respectivement.
\end{prop}

Ainsi le nombre de classes minimal est atteint pour l'anneau d'entiers: $h_\O$ est de l'ordre de $h_K c$. Cela amène à réviser la \og meilleure attaque\fg: il peut être intéressant de remonter à des courbes dont l'anneau d'endomorphismes est l'anneau d'entiers même en ayant au départ l'information que les anneaux d'endomorphismes sont les mêmes. Cela pousse à \emph{ne pas} profiter de l'augmentation du nombre de classes avec le conducteur: autrement dit, on cherche de bons paramètres dans le cas où l'ordre d'endomorphismes est $\O_K$.

D'autre part, on dispose de plusieurs résultats sur le nombre de classes d'un corps quadratique imaginaire. Tout d'abord, on peut facilement le majorer: si $K$ est un corps quadratique imaginaire de discriminant $D_K$, alors
$$h_K \leq \frac{1}{\pi} \sqrt{|D_K|}\, \ln(|D_K|).$$
On peut également dire des choses pour la minoration de ce nombre de classes. Le théorème de Brauer--Siegel donne une première indication, malheureusement non effective:
\begin{thm}[Brauer--Siegel pour les corps quadratiques imaginaires]
Soit $(K_i)_{i\in \N}$ une suite de corps quadratiques imaginaires de discriminants respectifs $D_i$ et de nombres de classes $h_i$. On suppose que
$|D_i| \underset{i\to\infty}{\vers} \infty.$
Alors
$$\frac{\log(h_i)}{\log(\sqrt{|D_i|})} \underset{i\to\infty}{\vers} 1.$$
\end{thm}

Ce résultat était en réalité connu avant les travaux de Brauer et Siegel, qui l'ont généralisé à des corps de nombres généraux. Stark a ensuite donné des versions effectives de ce théorème, ce qui amène à la conclusion (heuristique) suivante: si $K$ est un corps quadratique de discriminant $D_K$ et de nombre de classes $h_K$, et si $D_K$ est de l'ordre de $2^{512}$, alors $h_K$ est au moins de l'ordre de $2^{250}$; autrement dit, l'anneau d'entiers de $K$ fournit une bonne sécurité au vu de l'attaque précédente.

\v
Dans notre contexte, $K$ est obtenu comme le quotient $\Q[\pi]/(\pi^2  - t\pi + p)$. Le discriminant de l'ordre $Z[\pi]$ de $K$, engendré par l'endomorphisme de Frobenius, est $D = t^2 - 4 p$ et est donc en général de l'ordre de grandeur de $p$. Le discriminant $D_K$ de $K$ (qui est celui de $\O_K$) est alors l'entier sans facteur carré qui apparaît dans la décomposition $D = c^2 D_K,$ et $c$ est le conducteur de $\Z[\pi]$. Ainsi, déterminer le discriminant de $K$ revient à déterminer le facteur carré maximal de $D$.

Lorsque $D$ est un entier de l'ordre de 512 bits, factoriser $D$ est une opération que l'on ne peut pas faire chez soi, et qui implique l'utilisation de nombreuses machines durant plusieurs jours, avec l'utilisation d'algorithmes sophistiqués (notamment le \emph{crible algébrique}).

On a cependant tendance à penser qu'un entier choisi \og au hasard\fg\ a peu de chances de présenter un grand facteur carré; en moyenne, le conducteur de $\Z[\pi]$ sera donc, au pire, un petit entier. On peut donc espérer que ce sera le cas pour la courbe elliptique $E_0$ que nous finirons par choisir dans ce mémoire. On laisse donc ce problème de côté pour l'instant, et l'on choisira un nombre premier $p$ de 512 bits, de sorte que $D$ et donc $D_K$ soient de cet ordre également. Il faudrait vérifier que le conducteur de la courbe choisie est bien un petit entier, mais ce n'est pas fait ici.

\v
\textbf{Générateurs du groupe de classes.}
Soit $K$ un corps quadratique imaginaire de discriminant $D$, et soit $\O_K$ son anneau d'entiers. En admettant l'hypothèse de Riemann pour $K$, on peut montrer que les idéaux de $\O_K$ dont la norme est inférieure à $6\ln(|D|)^2$ engendrent le groupe de classe de $\O_K$.

Pour un nombre premier $p$ de 512 bits, utiliser cette borne amènerait à calculer des isogénies de degré près d'un million, ce qui est déraisonnable. En pratique, il est tout à fait possible que le nombre minimal de générateurs soit très inférieur à cette quantité: par exemple, un générateur suffit lorsque le groupe de classes est de cardinal premier. On considèrera donc que si $K$ un corps quadratique imaginaire, et $\O$ un ordre de $K$ de discriminant $D$ de notre ordre de grandeur, alors les idéaux inversibles de $\O$ dont la norme est inférieure à $\ln(|D|)$ engendrent $\Cl(\O)$, ou au moins un grand sous-groupe de celui-ci.

\v
Pour étudier la répartition des produits d'idéaux, on peut réfléchir à nouveau en termes de graphe d'isogénies. La répartition des produits d'idéaux dans le groupe de classes est alors directement reliée à la dispersion d'une certaine marche aléatoire dans ce graphe, et on sait que les propriétés d'\emph{expansion} de certains graphes garantissent la dispersion rapide des marches aléatoires. En admettant l'hypothèse de Riemann à nouveau, on peut montrer que le graphe d'isogénies satisfait ce type de propriété. 

Quelques définitions: pour un graphe non orienté $G$, $k$-régulier (c'est à dire que tout sommet admet $k$ voisins), on définit l'opérateur d'adjacence $A$ de la façon suivante:
$$A(f)\de x\longmapsto \sum_{y \text{ voisin de } x} f(y).$$
La fonction constante 1 est un vecteur propre de cet opérateur associé à la valeur propre $k$, dite valeur propre triviale; c'est la plus grande en valeur absolue. Le graphe $G$ est dit $\delta$-expanseur si les autres valeurs propres $\lambda$ de $A$ vérifient $|\lambda|\leq (1-\delta) k$. Lorsque $G$ est $\delta$-expanseur, les marches aléatoires dans $G$ de longueur $\frac{1}{\delta} \log(|G|)$ (où $|G|$ est le nombre de sommets) se répartissent uniformément dans le graphe, au sens où la probabilité de terminer dans un ensemble de sommets $S$ est au moins proportionnelle à $\#S$ (uniformément en $S$, bien sûr). Le résultat suivant est tiré de \cite{GRH}.

\begin{thm}
Soit $E_0/\F_p$ une courbe elliptique ayant multiplication complexe par $\O$. On se donne $B>2$. Soit $G$ le graphe dont les sommets sont les éléments de $\Ell_{\F_p}(\O)$, et dont deux sommets sont reliés par une arête s'ils sont reliés par une isogénie de degré inférieur à $(\log(4p))^B$. Alors $G$ est un graphe expanseur, dans le sens où les valeurs propres de l'opérateur d'adjacence vérifient
$$|\lambda| = O\left((\lambda_{\text{triv}} \log(\lambda_{\text{triv}}))^{\frac{1}{2} + \frac{1}{B}}\right).$$
Par conséquent, on peut donner une constante $C>0$ telle que lorsque $p$ est assez grand, une marche aléatoire de longueur au moins
$$C \frac{\log(h_\O)}{\log\log p}$$
termine dans un sous-ensemble fini $S$ de $\Ell_{\F_p}(\O)$ avec probabilité au moins $\frac{1}{2} \frac{\#S}{h_\O}$.
\end{thm}

L'hypothèse de Riemann intervient ici car les autres vecteurs propres de l'opérateur d'adjacence $A$, lorsque $G$ est le graphe de Cayley d'un groupe $H$, sont des caractères de $H$.

On peut faire deux remarques sur ce résultat. Premièrement, on ne veut certainement pas calculer des isogénies de degré $(\log p)^2$, mais seulement $\log(p)$, valeur qui permet d'atteindre une taille suffisante de l'espace de clés. Lorsque l'on durcit cette contrainte sur les idéaux considérés, le résultat ci-dessus ne s'applique plus, mais il semble plausible qu'une propriété d'expansion soit toujours vérifiée. Deuxièmement, on ne considèrera pas véritablement des marches alétoires uniformes dans le graphe ci-dessus: en effet, les isogénies de degré plus élevé sont en général plus difficiles à évaluer que les isogénies de petit degré. Ainsi, on aimerait effectuer plus de pas dans le graphe en utilisant les petits degrés que les grands, pour minimiser le coût total de l'opération.

On propose donc une hypothèse optimiste, mais intuitivement plausible, sur la répartition des produits d'idéaux dans notre cadre. 

\begin{hyp}
Avec les notations précédentes, soit $L$ un ensemble de nombres premiers, et $N$ le cardinal du sous-groupe de $\Cl(\O)$ engendré par les idéaux dont la norme figure dans $L$. Pour tout $\ell\in L$, on se donne une borne $M_\ell$ et l'on suppose que
\begin{itemize}
\item[•] Pour tout $\ell\in L$, on a $M_\ell \leq \ln(N)$;
\item[•] On a $\prod_{\ell\in L} (2M_\ell + 1) \geq N$.
\end{itemize}
Alors la plupart du temps, les produits $\prod {\frak l}^{a_\ell}$ pour $-M_\ell \leq a_\ell\leq M_\ell$ se répartissent de manière environ uniforme dans le sous-groupe, et $N$ est proche de $h_\O$.
\end{hyp}

Ainsi, on choisira un nombre premier $p$ de 512 bits, une liste $L$ de longueur environ 100 et des bornes $M_\ell$ vérifiant $\prod (2 M_\ell +1) \geq 2^{256}$. Montrer la sécurité de ce genre de paramètre semble impossible, mais utiliser les bornes prouvées données ci-dessus n'est pas raisonnable; on se contente donc de l'approche heuristique. Un élément semble montrer qu'une bonne répartition est possible: il ne peut pas exister de \og petite boucle\fg\ dans le gaphe d'isogénies, autrement dit de relation non triviale trop petite entre idéaux de petite norme. En effet, on en déduirait l'existence d'un endomorphisme des courbes elliptiques de petit degré qui n'est pas une multiplication scalaire, c'est à dire un élément de $\O$ non réel de petite norme: c'est impossible car $\O$ a un grand discriminant.

\v
Après cette discussion, on peut poser
$$\begin{aligned}
p =\ &120373407382088450343833839782228011370920294512701979230713977354082515866699 \\ &38291587857560356890516069961904754171956588530344066457839297755929645858769,
\end{aligned}
$$
un nombre premier de 512 bits congru à $-1$ modulo tous les nombres premiers compris entre 2 et 380. Notons que le produit de ces premiers est déjà un nombre de 509 bits. Afin de pouvoir utiliser le plus possible la méthode rapide de calcul d'isogénie, on souhaite maintenant exhiber une courbe elliptique $E/\F_p$ admettant:
\begin{itemize}
\item[•] Un modèle de Montgomery,
\item[•] Beaucoup de nombres premiers d'Elkies,
\item[•] Beaucoup de points de torsion rationnels de petit ordre.
\end{itemize}

Malheureusement, trouver une telle courbe est un problème difficile. Par exemple, étant donné un entier $C$ compris entre les bornes de Hasse, on ne sait pas exhiber une courbe de cardinalité $C$, bien que l'on puisse démontrer son existence. On peut toutefois faire un peu mieux que tester des courbes au hasard et observer leurs propriétés, comme on va le voir maintenant.

\subsection{Recherche de courbes}

On donne tout d'abord une stratégie (très) naïve pour trouver ce genre de courbe sur $k = \F_p$:
\begin{itemize}
\item[•] Choisir $a_4$ et $a_6$ au hasard dans $k$;
\item[•] Définir la courbe $E\de y^2 = x^3 + a_4x + a_6$;
\item[•] Regarder si $E$ a un modèle de Montgomery;
\item[•] Si oui, calculer le nombre de points $C$ de $E$ sur $k$ (ou sa trace $t$, de manière équivalente);
\item[•] Pour les premiers $\ell\leq 500$, regarder si $t^2 - 4p$ est un carré non nul modulo $\ell$ et si $C = 0 \mod\ell$;
\item[•] Estimer le temps nécessaire à un parcours dans le graphe d'isogénies en fonction des résultats;
\item[•] Conserver la meilleure courbe.
\end{itemize}

Pour que $E$ admette un modèle de Montgomery, il faut que $E$ admette un point rationnel de 2-torsion, donc que le polynôme $X^3 + a_4 X + a_6$ ait une racine dans $\F_p$. La courbe $E$ admet un modèle de Montgomery si, et seulement si, il existe une racine $\alpha$ de ce polynôme telle que le polynôme dérivé $3\alpha^2 + a_4$ soit un carré dans $\F_p$.

Bien sûr, le temps nécessaire pour effectuer un parcours dépend de l'implémentation choisie, et du choix du nombre de pas $M_\ell$ pour chaque nombre premier $\ell$ utilisé. On donnera des valeurs précises du temps d'exécution à la section suivante. Notons que même en connaissant le temps de calcul pour chaque premier, il est difficile de trouver les bornes $M_\ell$ optimales: en pratique, on se donne un certain temps $T$ que l'on accepte de dépenser pour chaque premier, on calcule alors la borne $M_\ell$ associée pour chaque $\ell$, et on obtient une majoration du temps de parcours lorsque le produit des $M_\ell$ devient assez grand.

\v
Pour déterminer les propriétés de $E[\ell]$ pour de petits premiers $\ell$, cette première stratégie nécessite de calculer la cardinalité d'une courbe sur $\F_p$. Pour cela, on utilise l'algorithme de Schoof, amélioré par Atkin et Elkies \cite{Schoof}. L'idée de cet algorithme polynomial est de calculer la trace $t$ modulo $\ell_i$, pour des nombres premiers $\ell_i$ tels que $\prod \ell_i > 4\sqrt{p}$. Connaissant les bornes de Hasse $|t|\leq 2\sqrt{p}$, on peut ensuite déterminer uniquement la trace dans $\Z$ modulo le théorème chinois.

Pour déterminer $t \mod \ell_i$, l'idée de Schoof est de regarder le Frobenius agissant sur l'espace $E[\ell_i](\bar{k})$. On a vu que c'est un endomorphisme de cet $\F_{\ell_i}$-espace vectoriel de dimension 2, dont la trace est $p\mod \ell_i$ et le déterminant $p$. On peut donc calculer le $\ell_i$-ième polynôme de division $\psi_i$ de la courbe $E$, et se demander pour quel élément $t$ l'égalité
$$(X^{p^2}, Y^{p^2}) - t (X^p, Y^p) + [p] (X, Y)$$
a lieu, où $(X, Y)$ est le point canonique de $E$ sur l'anneau $k[X, Y]$ quotienté par $\psi_i$ et l'équation de $E$. Notons que les sommes ci-dessus sont des sommes de poins sur une courbe elliptique. 

Cet algorithme est certes polynomial en $\log p$, mais de degré 8. L'amélioration proposée par Elkies est essentiellement celle qui a déjà été utilisée plus tôt dans ce document: lorsque $\ell_i$ est un nombre premier d'Elkies, on peut utiliser une équation modulaire pour travailler avec des polynômes de degré $\ell_i$ plutôt que $\ell_i^2$. En effet, on peut alors calculer une isogénie comme précédemment, calculer son noyau, puis calculer la valeur propre $v_i$ du Frobenius sur celui-ci. On saura alors que $t = v_i + \frac{p}{v_i} \mod \ell_i$.

L'amélioration proposée par Atkin concerne les autres premiers. Elle ne permet pas de déterminer directement $t \mod \ell_i$, mais détermine un petit ensemble de valeurs parmi lesquels $t \mod\ell_i$ doit se trouver. On ne peut alors plus appliquer directement le théorème chinois, et il faut effectuer une recherche supplémentaire pour déterminer la valeur exacte de $t$. Cette dernière recherche n'est pas un algorithme polynomial, mais cela améliore tout de même les performances en pratique.

L'algorithme de Schoof--Atkin--Elkies (SEA) a été bien étudié, et des implémentations matures sont aujourd'hui disponibles. Nous avons utilisé PARI \cite{PARI} pour ce calcul, à travers l'interface offerte par Sage \cite{Sage}. Pour le nombre premier $p$ choisi plus haut, une à deux minutes sont nécessaires environ pour calculer le cardinal d'une courbe sur $\F_p$.

\v
Revenons à la stratégie naïve donnée plus haut. Une première remarque est que l'on pourrait éviter le test du modèle de Montgomery, en tirant dès le début une courbe elliptique sous forme de Montgomery. On peut reformuler cela en disant que l'on dispose d'une paramétrisation (évidente) qui fournit des courbes sous forme de Montgomery.

Une question naturelle est alors la suivante: peut-on donner des paramétrisations qui donnent des courbes elliptiques ayant un point rationnel de $\ell$-torsion, pour certains nombres premiers $\ell$? On a vu que les courbes elliptiques munies d'un tel point sont des points de la courbe modulaire $Y_1(\ell)$; plus généralement on peut considérer des courbes $Y_1(N)$ où $N$ est un produit de nombres premiers distincts. La question est donc de savoir si l'on peut donner une paramétrisation de $Y_1(N)$. Cela n'est possible que si la courbe $X_1(N)$ est de genre zéro. Or, selon \cite{Coreens}:

\begin{prop}
Soit $N\geq 1$ un entier. Alors le genre de $X_1(N)$ est 
$$  g(N) = 
\begin{cases}
0 &\text{si}\ 1\leq N\leq 4 \\
\displaystyle 1 + \frac{N^2}{24} \prod_{\substack{p|N \\ p\ \mathrm{premier}}} \left(1 - \frac{1}{p^2}\right) + \sum_{\substack{d|N\\ d\geq 0}} \varphi(d)\varphi\left(\frac{N}{d}\right)& \text{sinon,}
\end{cases}$$
où $\varphi$ désigne la fonction indicatrice d'Euler.
\end{prop}

Par conséquent, le genre de $X_1(N)$ est nul lorsque $N\leq 10$ ou $N = 12$. Pour ces valeurs, on peut en effet exhiber des paramétrisations de $Y_1(N)$ pour lesquelles on est capable d'expliciter le morphisme
$$j \de Y_1(N) \vers \A^1$$
afin de pouvoir retrouver la courbe elliptique correspondante. Remarquons qu'utiliser une paramétrisation de $Y_1(12)$, par exemple, est strictement supérieur à utiliser directement le modèle de Montgomery: en effet, une courbe ayant un point de 4-torsion admet toujours un modèle de Montgomery qui est facile à déterminer, et on gagne en plus la présence d'un point rationnel de 3-torsion.

Bien sûr, plus le niveau $N$ est grand, plus l'information gagnée a de la valeur, ce qui pousse à utiliser des courbes modulaires pour lesquelles aucune paramétrisation n'est disponible. Lorsque l'on connaît une équation plane $g(x, y) = 0$ de $Y_1(N)$, on peut tout de même générer des points de cette courbe modulaire sur $\F_p$ par le procédé suivant: on tire $x\in \F_p$ au hasard, puis on recherche les racines du polynôme $g(x, Y)$ dans $\F_p$. Cela a de bonnes chances de donner un résultat, puisqu'un polynôme à coefficients sur $\F_p$ a de bonnes chances d'admettre une racine.

À quel point cette méthode est-elle plus avantageuse que la recherche de courbes au hasard? Lorsque l'on tire une courbe au hasard, elle a environ une chance sur $N$ d'admettre un point rationnel de $N$-torsion, et vérifier sa présence a un coût essentiellement proportionnel à $N$: en utilisant une équation modulaire, le calcul coûteux est la recherche de racines d'un polynôme de degré $N + 1$. Le coût total est donc essentiellement quadratique. En utilisant une équation de $Y_1(N)$, il faut en premier lieu la stocker, et ensuite chercher les racines d'un polynôme de degré $d$, le degré de l'équation $g$ en la variable $y$. Plus ce degré est faible, plus l'équation est petite et plus cette méthode est intéressante.

On arrive donc à la notion de \emph{gonalité}. On appelle gonalité d'une courbe algébrique $X$ sur un corps algébriquement clos $K$ le plus petit entier $k\geq 1$ tel qu'il existe une application rationnelle de degré $k$ de $X$ vers la droite projective. Autrement dit, c'est le plus petit entier $k$ tel qu'il existe une extension $K(X)/K(f)$ de degré $k$ pour un certain $f\in K(X)$. La gonalité mesure donc un défaut de rationalité, et les courbes de gonalité 1 sont exactement les courbes birationnelles à $\P^1$.

Lorsque $X$ est donné par une équation plane, la gonalité est majorée par le degré de l'équation en chacune de ses variables. Or, les résultats connus sur la gonalité de $X_1(N)$ sont décourageants.

\begin{prop}
La gonalité $k$ de $X_1(N)$ vérifie l'inégalité $\frac{21}{100}(g_1(N)-1)\leq k.$
\end{prop}

Dans ce résultat, tiré de \cite{gon}, $g_1(N)$ est le genre de $X_1(N)$, qui est quadratique en $N$ vu la proposition précédente. Une équation plane de $X_1(N)$ est donc nécessairement de degré quadratique en chacune de ses variables: le stockage seul de cet équation est trop coûteux pour envisager de l'utiliser pour de grandes valeurs de $N$.

La situation est sensiblement la même pour les courbes modulaires $X_0(N)$. Il s'agit donc de trouver le juste milieu entre les très petites valeurs, pour lesquelles utiliser les courbes modulaires est intéressant, et les valeurs plus élevées où le coût explose. En pratique, on peut utiliser les courbes $X_0(30)$ ou $X_1(17)$, de gonalité respectivement 2 et 4, et dont il faut alors calculer une équation.

\v
On peut apporter une autre amélioration à la stratégie naïve du début de cette section. En effet, au cours de l'algorithme SEA, on peut choisir de s'arrêter lorsque les valeurs de $t \mod \ell$ pour de petites valeurs de $\ell$ ne sont pas satisfaisantes, ou s'il y a trop peu de nombres premiers d'Elkies. On obtient alors l'algorithme de recherche que nous avons utilisé en pratique:
\begin{itemize}
\item[•] Générer des courbes candidates $E$ à l'aide d'une équation modulaire du bon niveau;
\item[•] Regarder si la courbe admet un modèle de Montgomery;
\item[•] Si oui, lancer l'algorithme SEA et s'arrêter s'il n'y a pas de torsion rationnelle pour de petits premiers;
\item[•] Une fois le cardinal calculé, estimer le temps nécessaire à un parcours dans le graphe d'isogénies en fonction des résultats;
\item[•] Conserver les meilleures courbes.
\end{itemize}

Le nombre de conditions à satisfaire dans l'étape 3 est déterminé empiriquement avec l'intuition suivante: l'algorithme sera le plus efficace s'il passe autant de temps à générer des candidats, vérifier les premières conditions et calculer la cardinalité en entier, sachant que moins de courbes sont concernées à mesure que l'on avance dans la liste ci-dessus. Par exemple, si l'on utilise la courbe modulaire $X_1(17)$,
on peut tester la présence de $\ell$-torsion rationnelle pour $\ell\in\{3, 5, 7, 11, 13\}$.

\v

Voyons maintenant comment calculer une équation pour ces courbes modulaires, ainsi que l'expression de la fonction $j$ sur ces courbes.

\subsection{Calcul d'une équation de $X_0(30)$}

L'espace des formes modulaires paraboliques de poids 2 et de niveau $\Gamma_0(30)$ est de dimension 3, et une base de cette espace est donnée par
$$\begin{aligned}
f_1 &= q - q^4 - q^6 - 2q^7 + q^9 + q^{10} - 2q^{11} + 2q^{12} + 2 q^{14} +\ldots\\
f_2 &= q^2 - q^4 - q^6 - q^8 + q^{10} + q^{12} + 3q^{16} + q^{18} - q^{20} - \ldots \\
f_3 &= q^3 + q^4 - q^5 - q^6 - 2q^7 - 2q^8 + q^{10} + 2q^{11} + 2q^{13} + \ldots
\end{aligned}$$
La forme $f_2$ est \emph{vieille} car elle s'écrit $g(2\tau)$ où $g$ est de niveau $\Gamma_0(15)$, ce qui n'est pas le cas de $f_1$ et $f_3$. Ces formes sont de poids 2, donc correspondent à des formes différentielles holomorphes sur la courbe $X_0(30)$. Le terme \emph{parabolique} signifie que ces formes s'annulent sur les pointes de $X_0(30)$, et on peut d'ailleurs lire l'ordre de leur zéro à la pointe $\infty$ sur la première puissance de $q$ apparaissant dans leur développement. On définit alors
$$ x = \frac{f_2}{f_3}, \quad y = \frac{dx/d\tau}{f_3}.$$
Les fonctions modulaires $x$ et $y$ sont de poids zéro, ce sont donc des fonctions méromorphes sur la surface de Riemann $X_0(30)$. On trouve une relation reliant $x$ et $y$ grâce à l'algèbre linéaire:
$$ y^2 =  x^8 + 6x^7 + 9x^6 + 6x^5 - 4x^4 - 6x^3 + 9x^2 - 6x +1. $$
Bien sûr, on ne peut tester cette égalité que sur un nombre fini de coefficients. Pour montrer que cette égalité a bien lieu, on peut utiliser l'argument suivant: la différence
$y^2 - P(x)$
est une fonction méromorphe sur $X_0(30)$ dont les pôles sont concentrés aux pointes, et les ordres de ses pôles sont bornés par une quantité explicite. De plus, le calcul montre que cette fonction admet un zéro à la pointe $\infty$ dont l'ordre est supérieur à, disons, 100. Si cette fonction était non nulle, elle aurait autant de zéros que de pôles, ce qui n'est pas le cas: elle est donc identiquement nulle.

Il reste à voir que $x$ et $y$ engendrent bien le corps des fonctions de $X_0(30)$. La courbe hyperelliptique $\Cl$ donnée par l'équation ci-dessus est de genre 3, et on dispose d'une application rationnelle donnée par les fonctions $x, y$:
$$X_0(30) \vers \Cl.$$
Or, le genre de $X_0(30)$ est également 3 (la dimension de l'espace des formes paraboliques ci-dessus), ce qui force cette application à être de degré 1 par la formule de Hurwitz.

Pourquoi avoir choisi ces fonctions $x$ et $y$ particulières? Pour le comprendre, on peut suivre le chemin inverse. Si l'on souhaite une équation de la forme $y^2 = f(x)$, alors les deux formes différentielles sur cette courbe hyperelliptique
$$\omega_3 =  \frac{dx}{y},\quad \omega_1 = \frac{x\, dx}{y}$$
engendrent l'espace des formes différentielles d'ordre 1, et ont un zéro d'ordre $g$ et 1 respectivement à l'infini, où $g$ désigne le genre de la courbe en question. On a donc choisi $f_3$ et $f_1$ de cette façon. C'est le raisonnement suivi dans la thèse de Galbraith \cite{Galbraith}.

Cette équation de $X_0(30)$ ne nous est utile que si l'on connaît l'expression de la fonction modulaire $j$ en fonction de $x$ et $y$, ce qui permet de retrouver une courbe elliptique à partir d'un point de la courbe modulaire $X_0(30)$. Comme on connaît la $q$-expansion de $j$, on peut faire de l'algèbre linéaire comme précédemment, et on trouve que $j$ peut s'exprimer en fonction de $x$ et $y$ par une fraction rationnelle de degré 36, de numérateur
$$ \begin{aligned}
 &11390626x^{36} + 136688232x^{35} + 615281004x^{34} + 1157841640x^{33} + 11390624x^{32}y + 385027410x^{32} \\
 &+ 102514896x^{31}y + 3899199936x^{31} + 307361808x^{30}y + 68878885856x^{30} + 164026064x^{29}y  \\
 &+ 461330789952x^{29} - 1401072176x^{28}y + 2024096667816x^{28} - 8580123696x^{27}y + 6544388993120x^{27} \\
 &- 54692903792x^{26}y + 16121144156304x^{26} - 290282276144x^{25}y + 30519217044192x^{25} \\
 &- 1110816170032x^{24}y + 44213136417480x^{24} - 3094506362864x^{23}y + 48084954873408x^{23} \\
 &- 6419993916592x^{22}y + 37447747660704x^{22} - 10015247114224x^{21}y + 18293916069056x^{21} \\
 &- 11672985621488x^{20}y + 2315694303708x^{20} - 9885812131312x^{19}y - 4389588785808x^{19} \\
 &- 5658074812208x^{18}y - 4676088269240x^{18} - 1674540316912x^{17}y - 3342328078608x^{17} \\
 &+ 354269688656x^{16}y - 2056348964388x^{16} + 760401538672x^{15}y - 837671190464x^{15} + 580040840752x^{14}y \\
 & + 34891362336x^{14} + 300727188592x^{13}y + 90210852288x^{13} + 129394057072x^{12}y + 159283365960x^{12} \\
  &- 333927056x^{11}y - 3950452128x^{11} - 2917217872x^{10}y + 33962685456x^{10} - 10863538576x^9y \\
  &- 10436233760x^9 + 144372848x^8y + 1058338344x^8 - 342561616x^7y - 907019328x^7 + 79366768x^6y \\
  &+ 209996384x^6 + 4157616x^5y - 19141824x^5 + 4326064x^4y + 4478610x^4 - 286544x^3y - 2160280x^3 \\
  &- 402192x^2y + 883116x^2 + 134064xy - 196248x - 14896y + 16354
\end{aligned}
$$
et de dénominateur
$$ \begin{aligned}
&x^{35} - 13x^{34} + 56x^{33} + 21x^{32} - x^{31}y - 1106x^{31} + 16x^{30}y + 4207x^{30} - 104x^{29}y \\
&- 4268x^{29} + 294x^{28}y - 16637x^{28} + 165x^{27}y + 64174x^{27} - 4184x^{26}y - 62529x^{26} \\
 &+ 14224x^{25}y - 131200x^{25} - 18436x^{24}y + 478757x^{24} - 18529x^{23}y - 645986x^{23} \\
 &+ 102680x^{22}y + 1004763x^{22} - 95688x^{21}y + 5962092x^{21} - 225622x^{20}y - 6407093x^{20} \\
 &+ 572493x^{19}y - 14435032x^{19} + 1166032x^{18}y + 11257145x^{18} - 1185120x^{17}y + 15043608x^{17} \\
 &- 3222648x^{16}y - 10395833x^{16} + 1481173x^{15}y - 6881894x^{15} + 4046208x^{14}y + 5269581x^{14} \\
 &- 1639608x^{13}y + 1177004x^{13} - 2375590x^{12}y - 1236463x^{12} + 1434343x^{11}y - 566142x^{11} \\
 &+ 409992x^{10}y + 341845x^{10} - 686768x^9y + 516016x^9 + 205628x^8y - 349313x^8 + 115085x^7y \\
 &- 61286x^7 - 88744x^6y + 105001x^6 + 9064x^5y - 35308x^5 + 4374x^4y + 6561x^4 - 729x^3y - 729x^3.
\end{aligned}
$$
La situation se complique légèrement, en revanche, si l'on veut connaître une équation de la courbe et les coordonnées du point de 2-torsion garanti par $X_0(30)$, par exemple. Connaître ce point de 2-torsion permet d'éviter une recherche de racine lorsque l'on se demande si la courbe elliptique obtenue admet un modèle de Montgomery. La difficulté est que les coordonnées $a_4$ et $a_6$ de l'équation de Weierstrass \og usuelle\fg, proportionnelles aux séries d'Eisenstein $E_4$ et $E_6$, ne sont pas des fonctions modulaires de poids zéro, mais $\lambda^2 a_4$ et $\lambda^3 a_6$ en sont pour toute fonction modulaire $\lambda$ de poids $-2$. Les coordonnées du point de 2-torsion sont alors un multiple de la série d'Eisenstein $E_2$, précisément $\frac{-\lambda}{6} E_2$. On renvoie à \cite{Elkies} pour les détails.

\subsection{Calcul d'une équation de $X_1(17)$}

Comme dans le cas de $X_0(30)$, on recherche une équation de $X_1(17)$ sous forme de courbe plane valable sur le corps des nombres complexes, et on espère que cette équation est à coefficients entiers et donne bien une équation de $X_1(17)_{\F_p}$ une fois réduite modulo $p$. Fixons $N = 17$; on recherche donc deux formes modulaires $f$, $g$ de niveau $\Gamma_1(N)$, et un polynôme $\Phi_N^{(f, g)}$ à coefficients entiers tel que
$$\Phi_N^{(f, g)}(f, g) = 0.$$

On suit ici l'article \cite{Baaziz}. On note $\wp(z, \Lambda_\tau)$ la fonction $\wp$ de Weierstrass associée au réseau $\Lambda_\tau = \Z\oplus\Z \tau.$

\begin{thm}
On définit deux fonctions méromorphes de la variable complexe :
$$f(\tau) = \frac{(\wp(1/N, \Lambda_\tau) - \wp(1/N, \Lambda_\tau))^3}{\wp'(1/N, \Lambda_\tau)^2}, \quad
g(\tau) = \frac{\wp'(2/N, \Lambda_\tau)}{\wp'(1/N,\Lambda_\tau)}.$$
Alors $f$ et $g$ sont deux fonctions modulaires pour $\Gamma_1(N)$, et engendrent le corps des fonctions de la courbe $X_1(N)$. Elles vérifient l'équation
$$\psi_N(1+g, f, f) = 0$$
où $\psi_N(a_1, a_2, a_3)$ est le $N$-ième polynôme de division de la courbe d'équation
$$y^2 + a_1xy + a_3y = x^3 + a_2x^2$$
évalué au point $(x, y) = (0, 0)$.
\end{thm}

Remarquons que le calcul de ce polynôme de division évalué en $(0, 0)$ est facile, en utilisant des formules récursives comme précédemment. Rappelons que le polynôme de division s'annule exactement sur les points affines de $N$-torsion de la courbe. Le lemme clé est le suivant:

\begin{lem}
Toute classe d'isomorphisme de courbes elliptiques $(E, P)$ où $P$ est un point d'ordre $N$ sur un corps $K$ algébriquement clos contient un représentant de la forme
$$E\de y^2 + ((1+g) x + f)y = x^3 + fx^2,\ P = (0,0).$$
avec $f\in K^\times,\ g\in K$. Ce représentant est unique à unique isomorphisme près.
De plus, si $K=\C$ et $(E, P) = (\C/\Lambda_\tau, 1/N)$, on a
$$f = \frac{(\wp(1/N, \Lambda_\tau) - \wp(1/N, \Lambda_\tau))^3}{\wp'(1/N, \Lambda_\tau)^2}, \quad  g = \frac{\wp'(2/N, \Lambda_\tau)}{\wp'(1/N,\Lambda_\tau)}.$$
\end{lem}
Ce lemme s'établit à l'aide de calculs sur les modèles de Weierstrass. L'ingrédient clé est que la courbe elliptique $\C/\Lambda_\tau$ est isomorphe à la courbe elliptique $y^2 = x^3 + g_2(\tau) x + g_3(\tau)$, l'isomorphisme étant donné par 
$z\mapsto (\wp(z), \wp'(z))$.

\begin{proof}[Démontration du théorème] On vérifie que les fonctions méromorphes
$$\tau\mapsto \wp(1/N, \Lambda_\tau),\quad \tau\mapsto\wp'(1/N,\Lambda_\tau),\quad \tau\mapsto\wp''(1/N, \Lambda_\tau)$$
sont des formes modulaires pour $\Gamma_1(N)$ de poids respectifs 2, 3 et 4. On montre alors que $f$ et $g$ sont des fonctions modulaires (de poids 0) pour $\Gamma_1(N)$ à l'aide des formules de duplication, donc des fonctions sur $X_1(N)_\C = \overline{\Gamma_1(N)\backslash \H}$. Si $a$ est un point affine de cette courbe modulaire, selon le lemme, $a$ admet un représentant $(E, P)$ qui s'écrit
$$E\de y^2 + ((1 + g(a))x + f(a))y = x^3 + f(a)x^2,\quad P = (0, 0).$$
$P$ étant un point de $N$-torsion, on a bien $\psi_N(1+g(a), f(a), f(a)) = 0$. \'Etant vraie sur un ouvert affine, cette égalité est vraie sur tout $X_1(N)$.

Il reste à montrer que $f$ et $g$ engendrent le corps des fonctions de $X_1(N)_\C$. Pour cela, il suffit de trouver un ouvert non vide $U$ de la surface de Riemann compacte $X_1(N)_\C$ tel que $(f, g)$ sépare les points sur $U$. Or, c'est le cas sur l'ouvert affine $Y_1(N)$, car le lemme donne un représentant d'un tel point en fonction uniquement de $f$ et $g$.
\end{proof}

Le fait que $f$ et $g$ engendrent le corps des fonctions montre que l'application rationnelle
$$X_1(N)_\C \overset{(f,\ g)}{\vers} \left\{(u, v)\in \P^2(\C) \de \psi_N(1+v, u, u) = 0\right\}$$
est un isomorphisme birationnel sur son image : on a donc obtenu un modèle plan de la courbe modulaire $X_1(N)$.

Vu les résultats précédents sur la gonalité de $X_1(N)$, on doit s'attendre à trouver des équations de degré de l'ordre de $N^2$ en chaque variable. En pratique, on peut s'arranger pour réduire ce degré d'un facteur constant. Une première étape est de remarquer que $\psi_N(1+g, f, f)$ possède un facteur $f^{v_N}.$ On peut déterminer cette valuation à partir des formules de récurrence, et on obtient
$$ v_N =
\begin{cases}
3k^2 & \text{si}\ N = 3k, \\
3k^2 + 2k &\text{si}\ N = 3k+1, \\
3k^2 + 4k + 1 &\text{si}\ N = 3k+2.
\end{cases}$$
En pratique, on calcule directement la quantité $f^{-v_N}\psi_N$, pour laquelle on peut donner une formule récursive semblable. De plus, si $M$ divise $N$, alors $\psi_M(1+g, f, f)$ divise $\psi_N(1+g, f, f)$ (si $MP = 0$, on a aussi $NP = 0$). On définit donc des polynômes $\Phi^{(f, g)}$ vérifiant :
$$\forall N\geq 1,\ f^{-v_N} \psi_N(1+g, f, f) = \prod_{M | N} \Phi_M^{(f, g)}.$$

On a alors $\Phi_N^{(f, g)}(f, g) = 0$. Une fois cette équation calculée, on peut tenter un changement de variables pour la simplifier. On pose ainsi successivement :
$$\begin{matrix} f = st(t-1), & g = s(t-1)\\
 s = q\frac{r+1}{r} - 1, & t = q(r+1) + 1.
 \end{matrix}$$
On peut aisément inverser ces changements de variables pour vérifier que $q$ et $r$ engendrent toujours le corps des fonctions de $X_1(N)$. L'équation obtenue reste à coefficients entiers, et est de plus petit degré que $\Phi^{(f, g)}$ après simplification; expliquer cette constatation expérimentale reste une question ouverte. Une fois que l'on connaît une équation de $X_1(N)$, on peut effectuer des changements de variables \og aléatoires \fg\ pour tenter d'en trouver une plus intéressante : ce travail est l'objet \cite{Sutheq}. Bien sûr, tout ceci est une affaire de facteurs constants, et on ne peut pas échapper à la minoration de la gonalité donnée précédemment. L'équation de $X_1(17)$ finalement utilisée s'écrit
$$y^4 + (x^3 + x^2 - x + 2)  y^3 + (x^3 - 3  x + 1)  y^2 - (x^4 + 2  x)  y + x^3 + x^2 = 0$$
et l'on retrouve la courbe elliptique avec les relations
$$s = -1 - \frac{y}{x^2 + x},\quad t = 1 + \frac{xy + y^2}{x^3 + x^2y + x^2}.$$
Ces équations sont certes plus élégantes que celles de $X_0(30)$ plus haut, mais elles sont de degré 4 en $y$ (ou en $x$), donc entraînent un coût en réalité plus élevé lors de leur utilisation. D'un autre côté, il est plus intéressant de connaître un point rationnel d'ordre 17 que de savoir que 2, 3 et 5 sont d'Elkies.

\v

Avec ces équations explicites de $X_0(30)$ et $X_1(17)$ et une base de données de polynômes modulaires, tous les ingrédients sont réunis pour pouvoir implémenter le cryptosystème de Couveignes--Rostovtsev--Stolbunov et la recherche de courbe initiale décrite plus haut.

\newpage

\section{Implémentation et performances}

\subsection{Systèmes utilisés}

Nous avons principalement utilisé deux systèmes de calcul formel lors de la mesure des performances de ce protocole cryptographique: Sage \cite{Sage} et Nemo \cite{Nemo}.

Sage (ou SageMath) est un logiciel de calcul formel libre et open-source sous licence GPL. Codé principalement dans le langage Python, sa particularité est de faire appel à de nombreuses bibliothèques C spécialisées dans divers domaines, comme GMP \cite{GMP} pour les opérations arithmétiques de base, Singular pour les polynômes à plusieurs variables, ou des systèmes plus vastes comme FLINT \cite{FLINT}, PARI \cite{PARI} ou NTL \cite{NTL}. Sage est extrêment vaste et couvre de nombreux domaines des mathématiques pures et appliquées, et se veut une alternative à des systèmes de calcul formel payants comme Maple, Mathematica, Magma ou Matlab. Du fait de sa richesse, Sage n'est malheureusement pas toujours bien contrôlé: son installation est parfois difficile (sous Windows notamment), et il souffre de nombreuses fuites de mémoire. D'autre part, le langage Python, bien que très facile d'accès, ne présente pas des performances optimales: cela pose problème lorsqu'une partie coûteuse du calcul n'est pas liée directement vers une bibliothèque C spécialisée.

Nemo est un système de calcul formel pour le langage de programmation Julia, libre et open-souce sous licence BSD. Julia est un langage de programmation jeune qui, contrairement à Python, dispose d'une compilation JIT (pour \og just in time\fg): cela implique qu'il est très facile d'écrire du code qui s'approche des performances de C. Dans le même temps, Julia est un langage de très haut niveau qui présente une syntaxe très naturelle similaire à celle de Python. Un dernier avantage est la présence d'un typage paramétrique, ce qui permet d'écrire un code très générique en gardant de bonnes performances. Par exemple, si ce code générique est finalement appelé sur des entiers 32 bits, le compilateur sera capable de le reconnaître et d'optimiser la gestion de la mémoire pour ce type de données précis.

Comme Sage, Nemo fait appel à des bibliothèques C spécialisées, principalement FLINT. Le système Nemo est très jeune, et est beaucoup plus restreint que Sage; certaines opérations \og basiques\fg\ ne sont pas toujours disponibles. C'est parfois un désavantage, mais cette simplicité permet de mieux contrôler le code que l'on écrit et de comprendre facilement ses performances.

\v
Sage contenait déjà des outils pour manipuler les courbes elliptiques, ce qui en faisait le système le plus accessible pour écrire les premiers tests. Il dispose également d'une base de données de polynômes modulaires facile d'accès, et d'une interface avec le logiciel de calcul formel PARI pour les calculs coûteux comme la recherche de racines d'un polynôme à coefficients dans $\F_p$. En revanche, il représente les courbes elliptiques sous forme de Weierstrass uniquement, alors qu'utiliser le modèle de Montgomery améliore significativement les performances de la multiplication scalaire. De plus, les opérations sur une courbe elliptique ne font pas appel à une bibliothèque C spécialisée, ce qui impacte les performances. 

Nous avons donc fait le choix de développer un module au-dessus de Nemo, afin de disposer de ces nouvelles fonctionnalités dans un langage, Julia, qui fournit des performances comparables à un code C natif tout en restant un langage de très haut niveau. Ce module a été utilisé pour la mesure de performances; à terme, il devrait devenir un module open-source contenant les grandes primitives nécessaires à la cryptographie sur courbes elliptiques, et pourra être utilisé dans un but pédagogique.

\subsection{Contenu}

Afin de pouvoir mesurer les performances du cryptosystème de Rostovtsev et Stolbunov, le module de courbes elliptiques pour Nemo dipose des fonctionnalités suivantes:
\begin{itemize}
\item[•] Définir des courbes elliptiques dans divers modèles, sur des corps ou des anneaux
\item[•] Définir des points sur ces courbes, les ajouter, les multiplier par un entier
\item[•] Calculer le $j$-invariant, et des isomorphismes entre différents modèles
\item[•] Définir des isogénies entre courbes et appliquer les formule de Vélu
\item[•] Avoir accès à une base de données de polynômes modulaires (cela fait appel à un autre module)
\item[•] Calculer le noyau d'une isogénie par l'algorithme de Bostan--Morain--Salvy--Schost
\item[•] Calculer la valeur propre du Frobenius sur un sous-groupe
\item[•] Étendre le corps de base d'une courbe elliptique
\item[•] Tirer aléatoirement un point sur une courbe elliptique sur un corps fini.
\end{itemize}

\v
À ce point du développement, on dispose des modèles de Weierstrass et de Montgomery, et le calcul d'isogénies fonctionne lorsque le degré de l'isogénie est impair. Afin que ce module puisse être utilisé par ailleurs, nous avons le projet d'y ajouter les choses suivantes:

\begin{itemize}
\item[•] D'autres modèles de courbes elliptiques, par exemple le modèle d'Edwards
\item[•] Le calcul de couplages entre deux points
\item[•] Le calcul de la cardinalité d'une courbe elliptique, avec un appel au logiciel PARI
\item[•] Un algorithme de calcul d'isogénie en petite caractéristique
\item[•] Et éventuellement d'autres choses en lien avec la cryptographie sur courbes elliptiques.
\end{itemize}

\v
Le code est en libre accès à l'adresse \url{github.com/JKieffer95/EllipticCurves.jl}. Pour l'utiliser, il suffit d'installer Julia 0.6 et Nemo, d'ajouter le dossier EllipticCurves dans le répertoire .../.julia/v0.6 et de demander: 
\begin{verbatim}
julia> Pkg.add("EllipticCurves")

julia> using EllipticCurves
\end{verbatim}
Le module est alors directement accessible.

\subsection{Mesure de performances}

On a vu qu'il est important de mesurer précisément les performances du calcul d'isogénies. Cela permet en effet de déterminer les performances globales du cryptosystème de Couveignes--Rostovtsev--Stolbunov pour une courbe initiale donnée. On peut ainsi d'une part choisir la meilleure courbe initiale lors de la recherche des paramètres, et d'autre part comparer les performances de ce protocole à d'autres échanges de clés existants. Toutes les mesures sont effectuées sur $\F_p$ où $p$ est le nombre premier déterminé précédemment, d'une taille de 512 bits; la machine utilisée dispose d'un processeur de 3.20GHz Intel Xeon.

Rappelons les deux méthodes de calcul à notre disposition pour calculer une isogénie de degré $\ell$ au départ d'une courbe $E$:
\begin{itemize}
\item[•] Lorsque la courbe elliptique admet un unique sous-groupe d'ordre $\ell$ dont les points sont définis sur une extension de degré $d$, on peut utiliser une multiplication scalaire sur $E$ pour trouver ce sous-groupe, puis utiliser les formules de Vélu;
\item[•] Dans le cas général, on utilise une équation modulaire de degré $\ell + 1$ dont on cherche les racines, on calcule le noyau de l'isogénie correspondante puis on regarde la valeur propre du Frobenius sur ce noyau.
\end{itemize}
Les étapes coûteuses sont alors respectivement
\begin{itemize}
\item[•] Une multiplication par $p^d$ (environ) sur $E(\F_{p^d})$,
\item[•] Une mise à la puissance $p$ modulo un polynôme de degré $\ell + 1$, puis une mise à la puissance $p$ modulo un polynôme de degré $\frac{\ell - 1}{2}$.
\end{itemize}
Ce sont ces deux opérations que l'on choisit de mesurer. Tout d'abord, on présente le coût de la multiplication scalaire dans trois cas: les opérations fournies par Sage (sur les modèles de Weierstrass), ainsi que notre implémentation en Nemo pour les modèles de Weierstrass et de Montgomery. Rappelons que l'on s'attend à trouver un coût quadratique en $d$.

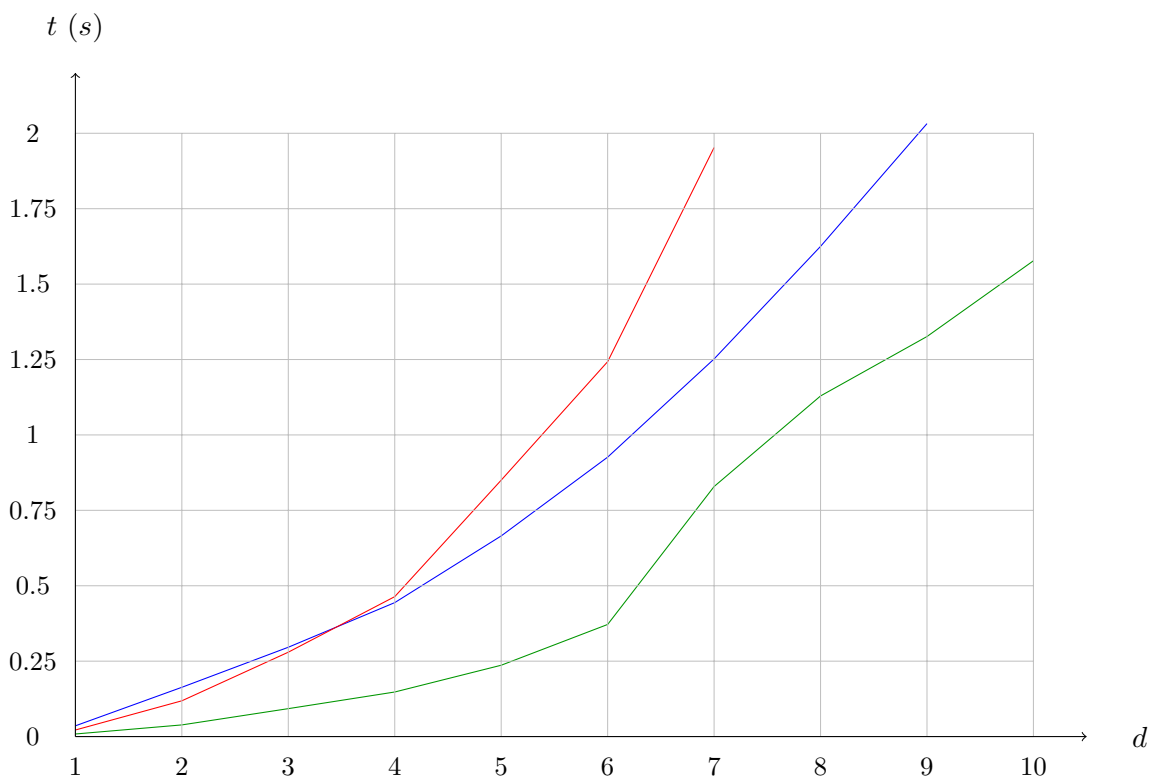
\begin{center}
\begin{tikzpicture}[x = 1.4cm, y=4cm]

\draw[blue] plot file {benchmark_sage_scalarmult.txt};
\draw[red] plot file {benchmark_weierstrass_scalarmult.txt};
\draw[darkgreen] plot file {benchmark_montgomery_scalarmult.txt};
\draw[color=gray!50, very thin] (1,0) grid[xstep = 1, ystep = 0.25] (10, 2);
\foreach \k in {1,2,...,10}
 \draw  (\k, -0.1) node  {\small\k};
 
\foreach \t in {0,0.25,...,2}
 \draw (0.6, \t) node {\small \t};
 
\draw[->] (1,0) -- (1, 2.2);
\draw[->] (1,0) -- (10.5, 0);
\draw (11, 0) node {$d$};
\draw (1, 2.35) node {$t\ (s)$};
\end{tikzpicture}

\small Performance pour le calcul de $[p^d] P$, avec $P \in E(\F_{p^d})$ en utilisant  \textcolor{blue}{Sage},
\textcolor{red}{le modèle de Weierstrass} et
\textcolor{darkgreen}{le modèle de Montgomery}.

\end{center}

Il en ressort qu'utiliser le modèle de Montgomery apporte un gain non négligeable par rapport au modèle classique de Weierstrass. Pour ce dernier, l'écart peut s'expliquer par le fait que le modèle de Weierstrass en Nemo utilise des opérations moins optimisées que celles de Sage. On voit également qu'il est beaucoup plus intéressant de disposer de points de $\ell$-torsion rationnels que sur une extension non triviale, même de degré 3 par exemple.

\v
D'autre part, on présente les performances du calcul de $X^p$ modulo $Q$ lorsque le degré de $Q$ varie, en utilisant la méthode évidente en Sage et en Nemo; pour comparaison, on montre également les performances de PARI pour calculer les racines d'un polynôme du même degré dans $\F_p$, qui réalise essentiellement le même calcul.

\begin{center}
\begin{tikzpicture}[x = 0.025cm, y=1.8cm]
\draw[blue] plot file {benchmark_sage_frob.txt};
\draw[darkgreen] plot file {benchmark_pari_frob.txt};
\draw[red] plot file {benchmark_nemo_frob.txt};
\draw[color=gray!50, very thin] (0,0) grid[xstep = 100, ystep = 0.5] (500, 5);
\foreach \k in {0,100,...,500}
 \draw  (\k, -0.2) node  {\small\k};
 
\foreach \t in {0,0.5,...,5}
 \draw (-20, \t) node {\small \t};
 
\draw[->] (0,0) -- (0, 5.3);
\draw[->] (0,0) -- (520, 0);
\draw (550, 0) node {$\ell$};
\draw (0, 5.5) node {$t\ (s)$};
\end{tikzpicture}

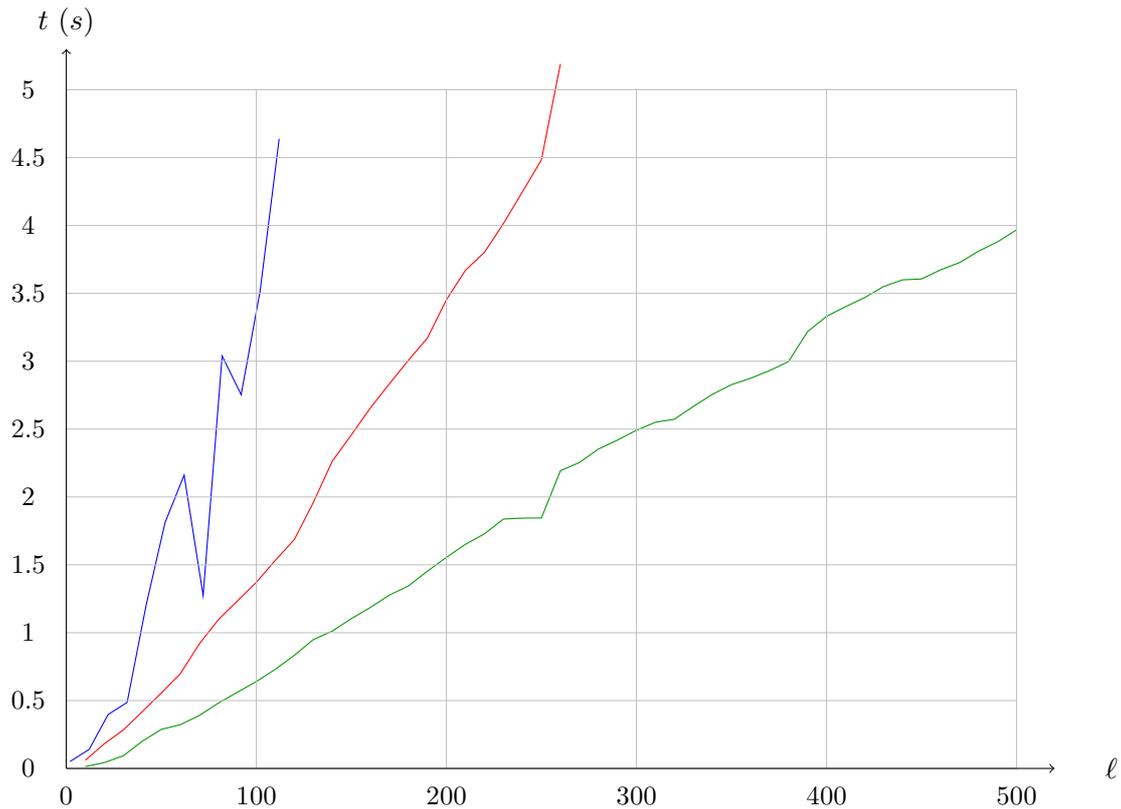
\small{Performance pour le calcul de $X^p \mod{Q(X)}$, où $\deg(Q) = \ell$ dans $\F_p$, avec 
\textcolor{blue}{Sage},
\textcolor{red}{Nemo} et 
\textcolor{darkgreen}{PARI}.}

\end{center}

On voit sur ce graphique qu'il est difficile d'atteindre les performances d'une bibliothèque C spécialisée.

\v
Afin de quantifier les performances du cryptosystème, on choisit la meilleure performance pour chacune de ces deux opérations, c'est à dire le modèle de Montgomery pour la multiplication scalaire et le logiciel PARI pour la recherche de racines. Nous pouvons maintenant donner une mesure de temps précise pour chacune des courbes initiales; c'est cette mesure que l'on utilise pour décider de la meilleure courbe initiale.

\subsection{Catalogue de courbes}

On peut maintenant donner les meilleures courbes obtenues par la méthode de recherche décrite précédemment. Après un premier tri, les meilleures candidates sont comparées grâce aux temps donnés ci-dessus.

Le calcul a été effectué par Sage sur les machines de l'équipe Grace d'Inria, à Saclay. Le temps de calcul total a été d'environ 17000 heures CPU (environ 2 ans, répartis sur différents c\oe urs de différents machines de l'équipe). Rappelons que
$$ \begin{aligned}
p = \ &120373407382088450343833839782228011370920294512701979230713977354082515866699 \\
&38291587857560356890516069961904754171956588530344066457839297755929645858769.
\end{aligned}$$

Voici donc les cinq meilleurs résultats, où l'on donne les paramètres $a, b$ de $\F_p$ définissant la courbe initiale 
$E\de y^2 = x^3 + ax + b$.
La meilleure courbe a été obtenue à partir d'un point de $X_1(17)$, les autres à partir d'un point de $X_0(30)$. 

\begin{enumerate}

\item[•]La courbe lauréate, donnée par les paramètres
$$\begin{aligned}
a =\ & 111164648488639530154045798728826155824357214700727018080372266722366593581458 \\
&61235530764279830938006566253824939581765595814774014481950202045501019370057,\\
b =\ & 79299948324245435156826219405909019495910509940412398234989651750360146055419 \\
&69152503725727477980546149790133670177545414851226849642510122057033563808087
\end{aligned}$$
donne un temps de parcours total de 241 secondes. Elle possède un point de $\ell$-torsion rationnel pour tout $\ell\in\{3, 5, 7, 11, 13, 17, 103, 523, 821, 947, 1723\}$. Son nombre de points sur $\F_p$ est
$$\begin{aligned}
&120373407382088450343833839782228011370920294512701979230713977354082515866700 \\
&85481138030088461790938201874171652771344144043268298219947026188471598838060.
\end{aligned}$$

\item[•]La courbe donnée par les paramètres
$$\begin{aligned}
a =\ & 43840601250479313001373029130703162545657631628386541829543985783667722196819 \\
& 83224703081195403305048751140227654592240178420097979549618534551394652828125,\\
b =\ & 15496029659544806200673781148304612576874915050506487020954702483761967446835\\
& 37981569065162806573553904845624037365860256326225753166777384039680905324745
\end{aligned}$$
donne un temps de parcours total de 255 secondes. Elle possède un point de $\ell$-torsion rationnel pour tout $\ell\in\{3, 5, 7, 11, 13, 617, 823, 911, 2741\}$. Son nombre de points sur $\F_p$ est
$$\begin{aligned}
&120373407382088450343833839782228011370920294512701979230713977354082515866701 \\
&38195889195780276859464571113614769417611266612000180056269739527141220119860.
\end{aligned}$$

\item[•]La courbe donnée par les paramètres
$$\begin{aligned}
a =\ & 29517794513388167079924232272984809284734463332655422100821259368004502646122 \\
&952457634401960 11377947505687906946702377331005567046387329096732076013487786,\\
b =\ & 20053142766210512206895831956067687278077475355845207144609905468409809112720 \\ &96530930100587341304378274081554750091335545148023305498915445491413837103327
\end{aligned}$$
donne un temps de parcours total de 268 secondes. Elle possède un point de $\ell$-torsion rationnel pour tout $\ell\in\{3, 5, 7, 11, 13, 19, 41, 53, 137, 463, 479, 1487\}$. Son nombre de points sur $\F_p$ est
$$\begin{aligned}
&120373407382088450343833839782228011370920294512701979230713977354082515866700 \\ 
&73203413637440471898330295081671080323874929304183417201801524565834682159200.
\end{aligned}$$

\item[•]La courbe donnée par les paramètres
$$\begin{aligned}
a =\ & 90951896883679247161305392056057501142984677621212827642316340355546948094470 \\
&24864291306024183700000129771938088507462877895751345076127980672694337232911,\\
b =\ & 53702533899051600924758507744833413409468531550522028426034775831774723715610 \\
&33304215623428708757733212555261802174294804948862439221808386692046723052122
\end{aligned}$$
donne un temps de parcours total de 273 secondes. Elle possède un point de $\ell$-torsion rationnel pour tout $\ell\in\{3, 5, 7, 11, 13, 37, 43, 127, 347, 457, 1289, 1621\}$. Son nombre de points sur $\F_p$ est
$$\begin{aligned} 
&120373407382088450343833839782228011370920294512701979230713977354082515866699 \\
&55535934625878401362947903345661723969818065903127482030480913154573514334480.
\end{aligned}$$

\item[•]Enfin, la courbe donnée par les paramètres
$$\begin{aligned}
a =\ & 87315792648930225008943587622481440451025081358306503063587707716608806180531 \\
& 79022484817139989036152987296557434567077843065574780789697405848884702097913,\\
b =\ & 63166216325315547106996156590059751380252981936070172805289816785905599938288 \\
& 29632834951338126458626870397768930797859281240868070581323508301210433040195\\
\end{aligned}$$
donne un temps de parcours total de 284 secondes. Elle possède un point de $\ell$-torsion rationnel pour tout $\ell\in\{3, 5, 7, 11, 13, 19, 23, 233, 359, 491, 631, 1481, 2579\}$. Son nombre de points sur $\F_p$ est
$$\begin{aligned}
&120373407382088450343833839782228011370920294512701979230713977354082515866700 \\
&82724204500287887625876543581163624210402478479308430356550349544933039974580.
\end{aligned}$$

\end{enumerate}

\v
Les temps annoncés sont des projections établies à partir des valeurs représentées sur les graphiques précédents; le cryptosystème \og entier\fg\ n'a pas été testé pour ces courbes initiales. Ce temps est celui d'un parcours dans le graphe d'isogénies, c'est à dire un calcul d'action pour un unique élément du groupe de classes sur une unique courbe. Ce calcul doit donc être réalisé deux fois si l'on souhaite effectuer un échange de clé. 

Avec ces courbes, effectuer un échange de clé en utilisant ce protocole nécessite donc une dizaine de minutes environ, ce qui n'est absolument pas compétitif comparé à d'autres protocoles cryptographiques. Environ 880 secondes sont nécessaires pour le calcul d'une action sans utiliser de \og bonne\fg courbe avec la même implémentation: on obtient donc tout de même un gain d'un facteur 4 environ.

\v

Pour clore ce document, on peut soulever plusieurs questions relatives à la sécurité ou aux performances de ce cryptosystème qui ne sont pas traitées ici:
\begin{itemize}
\item[•]Quelles seraient les performances du cryptosystème de Couveignes--Rostovtsev--Stolbunov si l'on parvenait à trouver une courbe parfaite? Serait-il intéressant de choisir un corps de base plus grand que nécessaire pour disposer de plus de marge pour trouver de bonnes courbes?
\item[•]Quel est le discriminant de la courbe finalement choisie? Peut-on vérifier au moins partiellement les propriétés de bonne répartition évoquées plus haut dans son anneau d'endomorphismes?
\item[•]Le choix d'une courbe ayant un grand nombre de points de torsion rationnels a une incidence sur la difficulté du problème du logarithme discret sur cette courbe elliptique: la courbe $E_0$ est de ce point de vue plus faible que la normale. Cela a-t-il une incidence sur la sécurité du protocole?
\item[•]À notre connaissance, la structure du groupe qui agit n'a pas d'incidence sur la sécurité du protocole, contrairement au protocole original de Diffie et Hellman. Peut-on construire une attaque basée sur la présence de facteurs cycliques dans le groupe de classes?
\item[•]Plus généralement, les propriétés spéciales des courbes choisies ont-elles une incidence sur la structure du groupe de classes?

\end{itemize}

\newpage

\bibliographystyle{plain}

\bibliography{Memoire}

\end{document}